\documentclass[11pt,reqno]{amsart}
\usepackage{}
\usepackage{amsfonts}
\usepackage[latin1]{inputenc}
\usepackage{bbm}

\usepackage{mathrsfs}

\usepackage{amsfonts,amssymb,amsmath,amsthm}
\usepackage{url}
\usepackage{enumerate}
\usepackage[pdftex,bookmarksnumbered]{hyperref}
\usepackage{graphicx,xcolor}
\newcommand{\ds}{\displaystyle}
\newtheorem{theorem}{Theorem}[section]
\newtheorem{lemma}[theorem]{Lemma}
\newtheorem{proposition}[theorem]{Proposition}
\newtheorem{corollary}[theorem]{Corollary}
\theoremstyle{definition}
\newtheorem{definition}[theorem]{Definition}
\newtheorem{remark}{Remark}
\numberwithin{equation}{section}

\newtheorem{example}{Example}

\usepackage[left=2.0cm, right=2.0cm, top=2.0cm, bottom=2.0cm]{geometry}

\usepackage{graphicx}%
\usepackage{color}
\usepackage{cite}
\usepackage{hyperref}
\hypersetup{
	colorlinks = true,%
	citecolor = blue,
	filecolor=red,%
	linkcolor = [rgb]{0.65,0.0,0.0},%
	anchorcolor = red,
	pagecolor = red,
	urlcolor= [rgb]{0.65,0.0,0.0}
	linktocpage=true,
	pdfpagelabels=true,
	bookmarksnumbered=true,
}

\DeclareMathOperator{\diam}{Diam}

\DeclareMathOperator{\loc}{loc}

\DeclareMathOperator{\di}{div}

\DeclareMathOperator{\lo}{loc}

\DeclareMathOperator{\dd}{{d}}

\DeclareMathOperator{\intt}{int}

\DeclareMathOperator{\supp}{supp}

\DeclareMathOperator{\nac}{nac}

\DeclareMathOperator{\AC}{\mathsf{AC}}

\DeclareMathOperator{\Lip}{Lip}

\DeclareMathOperator{\FAC}{\mathsf{FAC}}
\DeclareMathOperator{\NAC}{\mathsf{NAC}}
\DeclareMathOperator{\MAC}{\mathsf{MAC}}
\DeclareMathOperator{\BAC}{\mathsf{BAC}}
\DeclareMathOperator{\Vv}{\mathsf{pV}}
\DeclareMathOperator{\eV}{\mathsf{eV}}

 \DeclareMathOperator{\Bv}{\mathbf{v}}

\DeclareMathOperator{\e}{\mathfrak{e}}

\author{Fue Zhang}
\address{
Department of Mathematics\\
 Shihezi University\\
 832000 Shihezi, China}
\email{zhangfue@shzu.edu.cn}

\author{Wei Zhao}
\address{
Department of Mathematics\\
East China University of Science and Technology\\
200237 Shanghai, China}
\email{szhao\underline{ }wei@yahoo.com}

\begin{document}

\title[Absolutely continuous curves in Finsler-like spaces]{Absolutely continuous curves in Finsler-like spaces}

\begin{abstract}
The present paper is devoted to the investigation of absolutely continuous curves in asymmetric metric spaces induced by Finsler structures. Firstly, for asymmetric spaces induced by Finsler manifolds, we show that three different kinds of absolutely continuous curves coincide when their domains are bounded closed intervals. As an application, a universal existence and regularity theorem for gradient flow is obtained in the Finsler setting. Secondly, we study absolutely continuous curves in Wasserstein spaces  over   Finsler manifolds and establish the Lisini structure theorem in this setting, which characterize the nature of  absolutely continuous curves in  Wasserstein spaces in terms of dynamical transference plans concentrated on absolutely continuous curves in base Finsler manifolds. Besides,   a close relation between continuity equations and absolutely continuous curves in Wasserstein spaces is founded. Last but not least,  we also consider nonsmooth ``Finlser-like" spaces, in which case most of the aforementioned results remain valid.
Various model examples are constructed in this paper, which  point out genuine differences between the asymmetric and symmetric settings.

\end{abstract}
\maketitle


\section{Introduction}
In recent years, a lot of effort has been devoted to the study of absolutely continuous curves in  the symmetric setting, which reveals that such curves have  close  relations to the theories of Sobolev space, gradient flow and heat flow in the context of metric geometry (e.g., \cite{AGS,AGS2,AGS3,APS,HKST,GMS,GKO}).

Asymmetric metrics  often occur in nature (i.e., the symmetry $d(x,y) = d(y,x)$ is not assumed); a prominent example is the Matsumoto metric (see \cite{M}) describing the law of walking on a mountain slope under the action of gravity. In view of the preeminent role of absolutely continuous curves in geometry and analysis, it is natural as well as necessary to investigate  such curves in  asymmetric metric spaces.
Nonetheless, the lack of the
symmetric structure causes a significant difference and limited work has been done in  this  case. Therefore, the aim of the present paper is to investigate such curves  in ``Finsler-like" asymmetric metric spaces.

Many interesting asymmetric metric spaces can be induced by Finsler manifolds.
An typical example  is the \textit{Funk manifold} (e.g., \cite[Example~1.3.5]{Sh1}), which is  the $n$-dimensional Euclidean unit ball
$\mathbb{B}^n =\{ x \in \mathbb{R}^n \,|\, \|x\|<1\}$ ($n \ge 2$) endowed by a Finsler metric, called {\it Funk metric}, defined as $F:\mathbb{B}^n \times \mathbb{R}^n \rightarrow [0,\infty)$ by
\begin{equation}\label{Funckmeatirc}
 F(x,v)
 =\frac{\sqrt{\|v\|^2-(\|x\|^2\|v\|^2-\langle x,v \rangle^2)} +\langle x,v\rangle}{1-\|x\|^2},
 \quad x \in \mathbb{B}^n,\ v \in T_x\mathbb{B}^n=\mathbb{R}^n,
\end{equation}
where $\|\cdot\|$ and $\langle\cdot, \cdot\rangle$ denote
the Euclidean norm and inner product, respectively. The distance
function associated to $F$ is
\begin{equation}\label{distFunk}
 d_{F}(x_1,x_2) =\ln\Bigg(
 \frac{\sqrt{\|x_1-x_2\|^2-(\|x_1\|^2\|x_2\|^2-\langle x_1,x_2\rangle^2)}-\langle x_1,x_2-x_1\rangle}
 {\sqrt{\|x_1-x_2\|^2-(\|x_1\|^2\|x_2\|^2-\langle x_1,x_2\rangle^2)}-\langle x_2,x_2-x_1\rangle}
 \Bigg), \quad x_1,x_2 \in \mathbb{B}^n,
\end{equation}
Thus, the {\it Funk space} $(\mathbb{B}^n,d_F)$ is an asymmetric metric space, i.e., $d_F$ is
non-negative and verifies the triangle inequality, but not the symmetry; in particular,
\begin{equation}\label{radisulog2}
\lim_{\|x\|\rightarrow1^-}d_F(\mathbf{0},x)=\infty,\ \lim_{\|x\|\rightarrow1^-} d_F(x,\mathbf{0})=\ln 2.
\end{equation}
In fact,  $(\mathbb{B}^n,d_F)$  is an object which serves as a model structure where the symmetry fails. Moreover, we will see that the properties of absolutely continuous curves in the Funk space and in its Wasserstein space are surprisingly different from the ones in the symmetric case.

In the present paper, a map $\gamma$ from an interval $I\subset \mathbb{R}$ to an asymmetric space $(X,d)$ is called {\it $p$-forward absolutely continuous} (resp., {\it $p$-backward absolutely continuous}) for some $p\in [1,\infty]$ if there is a nonnegative $f\in L^p(I)$ such that
\begin{align*}
&d(\gamma(t_1),\gamma(t_2))\leq \int^{t_2}_{t_1}f(s){\dd}s \quad   \bigg(\text{resp.,\  }d(\gamma(t_2),\gamma(t_1))\leq \int^{t_2}_{t_1}f(s){\dd}s\bigg),
\end{align*}
for any $t_1,t_2\in I$ with $t_1\leq t_2$. And $\gamma$ is called  {\it $p$-absolutely continuous} if it is both $p$-forward and $p$-backward absolutely continuous. For convenience, we use $\FAC^p(I;X)$, $\BAC^p(I;X)$ and $\AC^p(I;X)$ to denote the classes of these three kinds of curves, respectively.
Although they coincide in the symmetric case  (cf. \cite[Definition 1.1.1]{AGS}), it is another story for asymmetric metric spaces. A simple example is
$X:=\mathbb{R}$ endowed with an asymmetric metric
\[
d(x,y):=y-x,\text{ if }y\geq x; \quad\quad d(x,y):=1,\text{ if }y< x.
\]
Thus, $ \FAC^p(I;X)\neq \BAC^p(I;X)$ for every $p\in [1,\infty]$ (see Example \ref{curverexa1} below for reasons). Furthermore, unlike the symmetric case, usually neither forward nor backward absolutely continuous curves can be extended to the closure of $I$ no matter whether $(X,d)$ is complete. For instance, consider
  the unit speed minimizing geodesic $\gamma$ defined on $I= (-\ln 2,0)$  in the Funk space $(\mathbb{B}^n,d_F)$ with
\begin{equation*}
\lim_{t\rightarrow -\ln 2^+}\gamma(t)=(-1,0,\ldots,0)\in \partial\mathbb{B}^n,\quad \lim_{t\rightarrow 0^-}\gamma(t)=\mathbf{0}\in \mathbb{B}^n.
\end{equation*}
Despite the forward completeness of $(\mathbb{B}^n,d_F)$, the $p$-forward absolutely continuous curve $\gamma$ cannot be extended to $t=-\ln 2$ for every $p\in [1,\infty]$ (see Example \ref{funcurve1} below for details).
In view of so many distinct phenomenons, it is meaningful and challenging as well to study absolutely continuous curves in the asymmetric setting.

Given a Finsler manifold $(M,F)$, the {\it reversibility} of a set $U\subset M$ (cf. \cite{R,Rademacher}) is defined as
\begin{equation}\label{resiblity}
\lambda_F(U):=\sup_{x\in U}\left( \sup_{y\in T_xM}\frac{F(x,-y)}{F(x,y)}   \right).
\end{equation}
Clearly, $\lambda_F(M)\geq 1$ with equality if  $(M,F)$ is  {\it reversible}. And every
  irreversible Finsler manifold $(M,F)$  induces an asymmetric metric space $(M,d_F)$ (cf. \cite{BCS,Sh1,KZ}).
 The first part of this article focuses on absolutely continuous curves in such a space (see Section \ref{sect1}). For a generic Finsler manifold  $(M,F)$,
 without the completeness assumption  we are able to prove that if $I$ is a bounded and closed interval, then
\begin{equation}\label{threearequ}
\FAC^p(I;M)= \BAC^p(I;M)=\AC^p(I;M).
\end{equation}
 In particular, the boundedness and closedness assumptions about $I$ are both necessary due to the counterexamples provided by the Funk space
 (see Theorem \ref{FACBAC} and Remark \ref{closureisbounded}). 
 Besides,  when $p=1$,  the curves in the classes \eqref{threearequ} are independent of the choice of Finsler structures and are the natural extension of absolutely continuous functions in classical real analysis (see Theorem \ref{newdefsabcon}), which is a Finsler version of \cite[Proposition 3.18]{AB}.

  It is well known that gradient flows govern a wide range of important evolution problems and hence, the related study
  has attracted remarkable attention. In spite of quite a number of contributions dealing
with the symmetric case (e.g., \cite{AGS,AGS2,AGS3,MS,Ograd,S,Vi}), as far as we know, there are only three papers  \cite{BM,RMS,OZ} concerning the asymmetric case, not to mention the Finsler case. Nevertheless, with the help of the aforementioned results, we can investigate the gradient flow in the context of Finsler manifolds. More precisely, given a continuous, subjective and  strictly increasing function $h:[0,\infty)\rightarrow [0,\infty)$, define the corresponding convex primitive function $\psi$ and the Legendre-Fenchel-Moreau transform $\psi^*$ as
\[
\psi(x):=\int^x_0 h(r){\dd}r,\quad \psi^*(y)=\sup_{x\geq 0}\left[ xy-\psi(x)  \right].
\]
Thus, for every bounded $C^1$-function $\phi:M\rightarrow \mathbb{R}$ and every point $x_0\in M$, we prove that there always exists a $C^1$-solution $\xi:[0,\infty)\rightarrow M$ satisfying the following {\it generalized gradient flow}
\begin{equation}\label{A}
\frac{\dd}{{\dd}t}(\phi\circ\xi)(t) = -\psi(F(\xi'(t)))-\psi^*{\left(F(\nabla(-\phi)|_{\xi(t)}\right)},\quad \xi(0)=x_0.
\end{equation}
In particular, Equation \eqref{A} contains the standard gradient flow $\xi(t)=\nabla(-\phi(\xi(t)))$ as well as the $p$-gradient flow considered in \cite{OZ}. Moreover,
 because of the singularity of Finslerian gradient operator $\nabla$, even if $\phi$ is smooth, the $C^1$-regularity of $\xi$ {\it cannot} be improved unless $\nabla(-\phi(\xi(t)))\neq 0$ for all $t>0$. See  Theorem \ref{existenceofgenreaflow} and Remark \ref{regularifinslerm} below for more details.

The second part of this paper is to devoted to absolutely continuous curves in Wasserstein spaces $(\mathscr{P}_p(M),W_p)$ over (irreversible) Finsler manifolds $(M,F)$ (see Section \ref{asbwassficcns}). Roughly speaking, $(\mathscr{P}_p(M),W_p)$ is an asymmetric metric space consisting of a collection of probability measures of the base manifold $M$.
Although Wasserstein space plays an important role in  modern geometric analysis (e.g., \cite{AGS2,LV,Sturm-1,Sturm-2}), up to now few results are available
in the literature concerning the asymmetric case (cf. \cite{KZ,OS}). The main difficulty is the incompatibility of the  forward and backward topologies. As an example, for every $p\in [1,\infty)$, in the Wasserstein space  $(\mathscr{P}_p(\mathbb{B}^n),W_p)$ over the Funk space, there always exist a sequence $(\mu_n)\subset \mathscr{P}_p(\mathbb{B}^n)$ and a probability measure $\mu\in \mathscr{P}(\mathbb{B}^n)\backslash \mathscr{P}_p(\mathbb{B}^n)$ such that
\begin{equation}\label{supphinfunk}
\lim_{k\rightarrow \infty}W_p(\mu,\mu_{n})=0,\quad  \lim_{n\rightarrow \infty}W_p(\mu_n,\mu)=\infty.
\end{equation}
See Example \ref{wpcontfunk} below for the constructions of $\mu_n$ and $\mu$. {This extreme incompatibility may cause the discontinuity of  partially absolutely continuous curves.}
In order  to avoid this, we make an additional assumption about the reversibility \eqref{resiblity}.
Observing that in  the Funk space the reversibility of the forward $r$-ball centered at $\mathbf{0}$ is a convex function in $r$ (cf. \cite[p.\,229]{KZ}),
 we focus on the Finsler manifolds whose reversibilities (restricted on forward balls centered at fixed points) satisfy a   concavity property.  Under this mid assumption, we succeed in extending the Lisini structure theorem (cf. \cite{Li,Li2}) to the Finsler setting (see Theorem  \ref{maintheorem}), in which case every $p(>1)$-forward absolutely continuous curve $\mu_t\in \FAC^p([0,1];\mathscr{P}_p(M))$ is a displacement interpolation, i.e., $\mu_t=(e_t)_\sharp\eta$, where $\eta$ is a dynamical transference plan (i.e., a probability measure on $C([0,1];M)$) concentrated on $\AC^p([0,1];M)$ and the ``speed" of $\mu_t$ satisfy
\[
|\mu'_+|_p(t):=\lim_{h\rightarrow 0^+}\frac{W_p(\mu_t,\mu_{t+h})}{h}= \left(\int_{C([0,1];M)} F^p(\gamma'(t)) \,{\dd}{\eta}(\gamma)\right)^{1/p}.
 \]
 for $\mathscr{L}^1$-a.e. $t\in (0,1)$. A similar result also holds for $p$-backward absolutely continuous curves. Moreover, provided $\mu_t\in \AC^p([0,1];\mathscr{P}_p(M))$, then this  structure theorem  even remains valid without the concavity assumption about reversibility due to the continuity of  $\mu_t$   in $(\mathscr{P}_p(M),W_p)$ (see Theorem \ref{basicfinterevers}).
 As an application, we show that every $\mu_t\in \AC^p([0,1];\mathscr{P}_p(M))$ can be interpreted as solutions of the {\it continuity equation}
\begin{equation}\label{contequa12}
\partial_t\mu_t+\di(\Bv_t\mu_t)=0
\end{equation}
for some vector field $\mathbf{v}_t$ on $M$ with
$\int_{[0,1]}\int_{ M} F^p(\Bv_t(x)){\dd}\mu_t{\dd}t<\infty$;
in particular, if the uniform constant (cf. \cite{E}) is finite, there is a unique $\mathbf{v}_t$ satisfying  \eqref{contequa12}  and
\[
|\mu'_+|_p(t)=\left(\int_{M}F^p(\Bv_t(x)){\dd}\mu_t(x)\right)^{1/p}
\]
for $\mathscr{L}^1$-a.e. $t\in (0,1)$ (see Theorem \ref{contequat1} and Corollary \ref{unformuqncon11}), which covers the results in Euclidean case \cite{AGS}, the Riemannian case \cite{Erb} and the compact Finsler case \cite{OS,OZ}.

Last but not least, we discuss briefly absolutely continuous curves in the framework of a special kind of nonsmooth asymmetric metric spaces, called {\it forward metric space} (see Section \ref{forresul1t}). Such a space is a  natural generalization of Finsler manifolds,  which possess many nice properties:
for instance, the Gromov-Hausdorff topology, the theory of curvature-dimension condition
developed by Lott, Sturm and Villani, and the theory of gradient flow initiated by Ambrosio,  Gigli and Saver\'e all can be generalized to such spaces (cf. \cite{KZ,OZ}). 

Although some constructions throughout the paper are similar to the symmetric case, peculiar differences
appear due to the character of the asymmetric metric spaces we are working on, which provide
the motivation and real flavor of the present work.

\smallskip

\noindent{\textbf{Acknowledgements.}}
The second author was supported by Natural Science Foundation of Shanghai
(No.\ 21ZR1418300).

\section{Preliminaries of asymmetric metric spaces}\label{asygecase}
As all the metric spaces in the present paper are asymmetric, we recall and introduce some basic definitions and properties of such spaces in this section. Also refer to \cite{CZ,D,KZ,Me,Me2,W}  for more details  (partly with different names).

\begin{definition}\label{generalsapcedef}
Let $X$ be a set and $d:X\times X\rightarrow [0,\infty)$ be a function on $X$. The pair $(X,d)$ is called  {an} {\it asymmetric metric space} if for any   $x,y,z \in X$:
\begin{itemize}

\item[(i)]  $d(x,y)\geq 0,
\mbox{ with equality if and only if } x=y;$

\smallskip

\item[(ii)] $ d(x,z)\leq d(x,y)+d(y,z).$

\end{itemize}
\end{definition}

Since the metric $d$
  could be asymmetric, there are two kinds of balls, i.e., forward and backward balls, respectively.
  More precisely, given any  $r>0$ and a point $x\in X$,  the {\it forward ball} $B^+_x(r)$ and the {\it backward ball} $B_x^-(r)$  of radius $r$ centered at $x$ are defined respectively as
\[
B^+_x(r):=\{y\in X| \, d(x,y)<r\},\quad B^-_x(r):=\{y\in X| \, d(y,x)<r\}.
\]

Let $\mathcal {T}_+$ (resp., $\mathcal {T}_-$) denote by the {\it forward  topology} (resp., {\it backward topology})  which is induced by forward balls (resp., backward balls) and let $\hat{\mathcal {T}}$ be the {\it symmetrized topology} that is induced by both forward open balls and backward ones. Owing to \cite[Section 3]{Me} and \cite[Section 2]{OZ}, we have the following result.

\begin{theorem}\label{asymmetrtopol}
Let $(X,d)$ be an asymmetric metric space. Then the following statements are true:
\begin{itemize}

\item[(i)]  $d:X\times X\rightarrow [0,\infty)$ is continuous under $\hat{\mathcal {T}}\times \hat{\mathcal {T}}$ and $(X,d)$ is a Hausdorff space;

\item[(ii)] the symmetrized topology $\hat{\mathcal {T}}$ is exactly the one induced by the symmetrized metric
\begin{equation}\label{symmmetricde}
\hat{d}(x,y)=\frac12[d(x,y)+d(y,x)].
\end{equation}

\item[(iii)] a sequence $(x_n)$ converges to $x$ in $(X,\hat{\mathcal {T}})$ if and only if both $d(x,x_n)\rightarrow 0$ and $d(x_n,x)\rightarrow 0$.

\end{itemize}
\end{theorem}

We now  recall the definition of completeness in the asymmetric setting.
\begin{definition}
Let $(X,d)$ be {an} asymmetric metric space.

\begin{itemize}
\item A sequence $(x_n)$ in $X$ is called
a {\it forward} (resp., {\it backward}) {\it Cauchy sequence} if, for each $\epsilon>0$, there exists
$N>0$ satisfying when $n\geq m>N$, then $d(x_m,x_n)<\epsilon$ (resp.,
$d(x_n,x_m)<\epsilon$).

\smallskip

\item  $(X, d)$ is called {\it forward} (resp., {\it backward}) {\it complete} if
every forward (resp. backward) Cauchy sequence converges in
$X$ with respect to $\hat{\mathcal {T}}$. And $(X,d)$ is called {\it complete} if it is both forward and backward complete.

\end{itemize}
\end{definition}
These two kinds of completeness are not equivalent. For example, the Funk space  \eqref{distFunk} is forward complete but not backward complete (cf. \cite{Sh1}). But if an asymmetric metric space $(X,d)$ is forward complete, then the corresponding {\it reverse space} $(X,\overleftarrow{d})$ is backward complete, where the {\it reverse metric} $\overleftarrow{d}(x,y)$ is  defined by $\overleftarrow{d}(x,y):=d(y,x)$. For this reason, in what follows we focus on forward complete asymmetric metric spaces.

We turn to the concept of absolutely continuous curve in the framework of asymmetric metric space.
Inspired by \cite{AGS2,RMS,OZ},  the following definition is introduced.

\begin{definition}\label{forwardabcontinc}  Let $(X,d)$ be an asymmetric metric space and let $I$ be an interval of $\mathbb{R}$.  Given $p\in [1,\infty]$, a map $\gamma:I\rightarrow X$ is said to be {\it $p$-forward absolutely continuous} (resp., {\it $p$-backward absolutely continuous})
if there exists a non-negative function $f\in L^p(I)$ such that
\begin{align}
&d(\gamma(t_1),\gamma(t_2))\leq \int^{t_2}_{t_1}f(s){\dd}s \quad   \bigg(\text{resp.,\  }d(\gamma(t_2),\gamma(t_1))\leq \int^{t_2}_{t_1}f(s){\dd}s\bigg),\label{moredfabsc}
\end{align}
for any $t_1,t_2\in I$ with $t_1\leq t_2$. The classes of $p$-forward absolutely continuous curves and $p$-backward absolutely continuous curves defined on $I$ are denoted by $\FAC^p(I;X)$ and $\BAC^p(I;X)$ respectively.
A curve $\gamma$ is called {\it $p$-absolutely continuous} if $$\gamma\in \FAC^p(I;X)\cap \BAC^p(I;X)=:\AC^p(I;X).$$
In particular, $ \FAC^1(I;X)$,  $\BAC^1(I;X)$ and $\AC^1(I;X)$ are denoted  by $ \FAC(I;X)$,  $\BAC(I;X)$ and $\AC(I;X)$, respectively.
\end{definition}

It is easy to see that $\gamma\in \AC^p(I;X)$ is continuous in $(X,d)$ with respect to $\hat{\mathcal {T}}$, but this is not true for partially  absolutely continuous curves.
The following example emphasizes the contrast between these two kinds of curves.
\begin{example}\label{curverexa1}
Let $X=\mathbb{R}$ endowed by the following asymmetric metric
\begin{align*}d(x,y):=\left\{
\begin{array}{lll}
y-x, && \text{ if }y\geq x;\\
\\
1, && \text{ otherwise. }
\end{array}
\right.
\end{align*}
Thus, $ \FAC^p([0,1];X)\neq \BAC^p([0,1];X)$. In fact, consider $\gamma(t):=t$, $t\in [0,1]$. It is easy to see
\begin{equation}\label{forbaccurvenotequ}
d(\gamma(t_1),\gamma(t_2))= \int^{t_2}_{t_1}1{\dd}s,\quad d(\gamma(t_2),\gamma(t_1))=1,
\end{equation}
for any $[t_1,t_2]\subset [0,1]$. Hence, $\gamma\in \FAC^p([0,1];X)$ but $\gamma\notin \BAC^p([0,1];X)$. In particular, $\gamma$ is discontinuous with respect to $ \hat{\mathcal {T}}$.
\end{example}

In a complete symmetric metric space, absolutely continuous curves defined on open intervals can be extended naturally to the closure of domains due to the uniform continuity (cf. \cite{AGS}). However, this is not true in the asymmetric case. The following example indicates that a forward absolutely continuous curve in a forward complete asymmetric metric space  can be extended forwardly but not backwardly.
 \begin{example}[\cite{OZ}]\label{funcurve1}
In the Funk space defined by \eqref{distFunk},
consider a unit speed minimizing geodesic $\gamma:(-\ln 2,0) \rightarrow \mathbb{B}^n$
such that
\begin{equation}\label{convergebn}
\lim_{t\rightarrow -\ln 2^+}\gamma(t)=(-1,0,\ldots,0)\in \partial\mathbb{B}^n\quad \lim_{t\rightarrow 0^-}\gamma(t)=\mathbf{0}\in \mathbb{B}^n.
\end{equation}
Here, the limits are defined by the standard topology of $\mathbb{R}^n$.
Thus, $\gamma \in \FAC((-\ln 2,0);\mathbb{B}^n)$ with $f \equiv 1$ in (\ref{moredfabsc}) because of the standard theory of geodesics in Finsler geometry (cf. \cite{Sh1}). Moreover, since $\hat{\mathcal {T}}$ is the standard topology of $\mathbb{B}^n$ (cf. \cite{BCS,Sh1}), it follows by \eqref{convergebn} that $\gamma$ can be extended at $t=0$  but not at $t=-\ln 2$. In fact, $\lim_{t \to -\ln 2^+} d_F(\mathbf{0},\gamma(t))=\infty$ and hence, $\gamma\notin \BAC((-\ln 2,0);\mathbb{B}^n)$.
\end{example}

In the sequel, let $\mathscr{L}^1$ denote the Lebesgue measure on $\mathbb{R}$. Thus,
it follows from \cite[Proposition 2.2]{RMS} that the ``partial" metric derivative of an absolutely continuous curve always exists for $\mathscr{L}^1$-a.e. point in the domain.
\begin{proposition}\label{baspeedforback}Let $(X,d)$ be an asymmetric metric space and let $I$ be an interval of $\mathbb{R}$.
Given $p\in [1,\infty]$, for any $\gamma\in \FAC^p(I;X)$ (resp., $\gamma\in \BAC^p(I;X)$), the forward (resp., backward) metric derivative
\begin{align*}
|\gamma'_+|(t)&:=\lim_{h\rightarrow 0^+}\frac{d(\gamma(t),\gamma(t+h))}{h}=\lim_{h\rightarrow 0^+}\frac{d(\gamma(t-h),\gamma(t))}{h}\\
 \bigg(\text{resp.,\  }|\gamma'_-|(t)&:=\lim_{h\rightarrow 0^+}\frac{d(\gamma(t+h),\gamma(t))}{h}=\lim_{h\rightarrow 0^+}\frac{d(\gamma(t),\gamma(t-h))}{h}\bigg),
\end{align*}
 exists for $\mathscr{L}^1$-a.e. $t\in \intt(I)$. Furthermore, $|\gamma'_+|$ (resp., $|\gamma'_-|$) is in $L^p(I)$ and satisfies
\begin{align*}
&d(\gamma(t_1),\gamma(t_2))\leq \int^{t_2}_{t_1}|\gamma'_+|(s){\dd}s \quad   \bigg(\text{resp.,\  }d(\gamma(t_2),\gamma(t_1))\leq \int^{t_2}_{t_1}|\gamma'_-|(s){\dd}s\bigg),
\end{align*}
for any $t_1,t_2\in I$ with $t_1\leq t_2$.

\end{proposition}
On account of \eqref{forbaccurvenotequ}, a  forward (resp., backward) absolutely continuous curve may not have the backward (resp., forward) metric derivative.
We also emphasize that the notation $|\gamma'_\pm|$ notwithstanding, the curve $\gamma\in \AC^p(I;X)$ may be not  differentiable in the usual sense, even if $(X,d)$ is  a Minkowski normed space.
\begin{example}
Let $X:=L^1([0,1])$ and let $d(f,g):=\|g-f\|_{L^1}+\omega\int_{0}^1 (g-f)(s){\dd}s$, where $\omega\in (-1,1)$ is a constant. Thus, $(X,d)$ is an asymmetric metric space induced by the Minkowski norm $[f]:=\|f\|_{L^1}+\omega\int_{0}^1 f(s){\dd}s$.
Set
 $\gamma(t):={\textbf{1}_{[0,t]}}$ for $t\in [0,1]$, which is a family of  indicator functions. It is easy to check
\[
d(\gamma(t_1),\gamma(t_2))=|t_2-t_1|+\omega(t_2-t_1),\quad \forall\,t_1,t_2\in [0,1],
\]
which implies $\gamma\in \AC([0,1];X)$. However, $\gamma$ is not differentiable in the usual sense but
\[
|\gamma'_+|(t)=1+\omega,\quad |\gamma'_-|(t)=1-\omega,\quad \forall\,t\in (0,1).
\]
\end{example}

\section{Absolutely continuous curves in Finsler manifolds} \label{sect1}
Asymmetric metric spaces induced by
Finsler manifolds enjoy many interesting properties. This section is devoted to the  investigation of absolutely continuous curves  in such metric spaces.

\subsection{Finsler manifolds}
In this subsection we recall some definitions and properties from Finsler geometry; for details see  \cite{BCS,R,Shen2013,Sh1,Rademacher}, etc.

 Let $M$ be an $n(\geq2)$-dimensional connected
 smooth manifold without boundary and $TM=\bigcup_{x \in M}T_{x}
M $ be its tangent bundle. The pair $(M,F)$ is a \textit{Finsler
	manifold} if the continuous function $F:TM\to [0,\infty)$ satisfies
the following conditions:

\begin{itemize}

\item[(a)] $F\in C^{\infty}(TM\setminus\{ 0 \});$

\item[(b)] $F(x,\lambda y)=\lambda F(x,y)$ for all $\lambda\geq 0$ and $(x,y)\in TM;$

\item[(c)] $g_{ij}(x,y)=[\frac12F^{2}%
]_{y^{i}y^{j}}(x,y)$ is positive definite for all $(x,y)\in
TM\setminus\{ 0 \}$ where $F(x,y)=F(y^i\frac{\partial}{\partial x^i}|_x)$.

\end{itemize}

Condition (c) indicates the strong convexity of $F$, i.e., $F(x,y_1+y_2)\leq F(x,y_1)+F(x,y_2)$ for any $y_1,y_2\in T_xM$, with equality if and only if $y_1=\alpha y_2$ for $\alpha\geq 0$. Besides, Condition (a) is natural because
  $F(x,y)=\sqrt{g_{ij}(x,y)y^iy^j}$ for $(x,y)\in TM\backslash\{0\}$ and $(g_{ij})$  cannot be defined at $y=0$ unless   it is a Riemannian metric. In particular, $F$ is called {\it Riemannian} if $(g_{ij})$ is a Riemannian metric.

The  {\it reverse Finsler metric} of $(M,F)$ is defined as $\overleftarrow{F}(x,y):=F(x,-y)$. Then $(M,\overleftarrow{F})$ is also a Finsler manifold, which is called the {\it reverse Finsler manifold}. 

The {\it reversibility}  of a subset $U\subset (M,F)$ is defined as follows:
\[
\lambda_F(U):=\sup_{x\in U}\sup_{y\in T_xM}\frac{F(x,-y)}{F(x,y)}.
\]
Thus, $\lambda_F(M)\geq 1$ with equality if $F$ is {\it reversible} and moreover, $\lambda_F(x):=\lambda_F(\{x\})$ is a continuous function.

The {\it Legendre transformation} $\mathfrak{L} : TM \rightarrow T^*M$ is defined
by
\begin{equation*}
\mathfrak{L}(X):=\left \{
\begin{array}{lll}
 g_X(X,\cdot), & \ \ \text{ if } X\neq0; \\
 \\
0, & \ \ \text{ if } X=0.%
\end{array}
\right.
\end{equation*}
In particular,  $\mathfrak{L}:TM\backslash\{0\}\rightarrow T^*M\backslash\{0\}$ is a diffeomorphism and $F^*(\mathfrak{L}(X))=F(X)$, for any $X\in TM$.
Now let $f : M \rightarrow \mathbb{R}$ be a $C^1$-function on $M$; the
{\it gradient} of $f$ is defined as $\nabla f = \mathfrak{L}^{-1}({\dd}f)$. Thus,  ${\dd}f(X) = g_{\nabla f} (\nabla f,X)$. For a non-Riemannian Finsler metric, $\nabla$ is usually nonlinear, i.e., $\nabla(f+h)\neq\nabla f+\nabla h$.

The {\it dual metric} $F^*$ of $F$  is
defined as
\begin{equation*}\label{dualFinslermetric}
F^*(x,\xi):=\underset{y\in T_xM\backslash \{0\}}{\sup}\frac{\langle y, \xi\rangle}{F(x,y)},  \quad
\forall\, \xi\in T_x^*M,
\end{equation*}
where $\langle y,\xi\rangle$ is the canonical pairing between $T_xM$
and $T^*_xM$. Thus, $F^*$ is a Finsler metric on $T^*M$ with
\begin{equation}
\langle y,\xi\rangle\leq F(x,y)F^*(x,\xi),\quad \forall\,y\in T_xM,\ \xi\in T^*_xM, \label{tangninnerprot}
\end{equation}
with equality if and only if $y=\alpha\mathfrak{L}^{-1}(\xi)$ with $\alpha\geq 0$.

A smooth curve $\gamma: [0,\infty)\rightarrow M$ is called a {\it geodesic} if it satisfies  the following ODE
$$\frac{{\dd}^2}{{\dd}t^2}\gamma^i(t)+ 2G^i
\left(\frac{\dd}{{\dd}t}\gamma(t)\right)=0,$$
where
\[
G^i(y):=\frac14 g^{il}(y)\left\{2\frac{\partial g_{jl}}{\partial x^k}(y)-\frac{\partial g_{jk}}{\partial x^l}(y)\right\}y^jy^k,\quad (g^{ij}):=(g_{ij})^{-1}.
\]

In the sequel, we study the length structure of a  Finsler manifold $(M,F)$.
Let $\mathcal {A}_\infty([0,1];M)$ denote the {\it class of piecewise smooth curves defined on $[0,1]$} (with monotonous parameterizations).
Given $\gamma\in \mathcal {A}_\infty([0,1];M)$, its \textit{length}   is defined as
\[
L_F(\gamma):=\int^1_0 F(\gamma'(t)){\dd}t.
\]
Define the {\it distance function} $d_F:M\times M\rightarrow [0,\infty)$ by
\begin{equation}
d_{F}(p,q):=\inf\{L_F(\gamma)\,|\,\gamma\in \mathcal {A}_\infty([0,1];M),\
\gamma(0)=p,\ \gamma(1)=q\}. \label{distanceinFinslersetting}
\end{equation}
Thus, $(M,d_F)$ is an asymmetric metric space.
In particular,
$d_F(p,q)\neq d_F(q,p)$  unless $F$ is reversible. Moreover, the forward topology coincides with the backward topology, which is exactly the original topology of the manifold.

A Finsler manifold $(M,F)$ is said to be {\it  forward}   (resp., {\it backward}) {\it complete} if  $(M,d_F)$ is so, in which case every geodesic defined on $[0,1]$  (resp., $[-1, 0]$) can be extended to $[0, \infty)$ (resp. $(-\infty, 0]$).
According to  the Hopf-Rinow theorem (cf. \cite{BCS,Sh1}), the closure of a forward ball (with a finite radius) is always compact if $(M,F)$ is forward complete. However, this is not true for backward balls. For instance,   the Funk manifold defined by \eqref{Funckmeatirc} is forward complete, but the closure of a backward ball centered at $\mathbf{0}$ with any radius $r\geq \ln 2$ is  noncompact because such a backward ball is the total manifold $\mathbb{B}^n$ (see \eqref{radisulog2}).

By means of $d_F$, one can define
 a new length structure  as
\[
L_{d_F}(\gamma):=\sup\left\{  \sum_{i=1}^n d_F(\gamma(t_{i-1}),\gamma(t_i))\,\Big|\, n\in \mathbb{N},\ 0=t_0<\cdots <t_n=1 \right\},\quad \forall\,\gamma\in C([0,1];M).
\]
A standard argument (cf. \cite{BM,KZ}) yields that $d_F$ is an intrinsic metric and hence,
\[
L_F(\gamma)=L_{d_F}(\gamma),\quad \forall\,\gamma\in \mathcal {A}_\infty([0,1];M).
\]

\subsection{Absolutely continuous curves in a Finsler manifold}
In this subsection, we study the absolutely continuous curves
in the asymmetric metric space $(M,d_F)$ induced by a (irreversible) Finsler manifold $(M,F)$ {\it without the completeness assumption}.
Inspired by \cite{AB},
we first show that  the nature of such curves can be characterized in terms of the differential structure of a manifold or the absolute continuity of their length (see Theorem \ref{newdefsabcon} below).

\begin{definition}\label{abthreekinds}Let $(M,F)$ be a  Finsler manifold and let $I$ be an interval of $\mathbb{R}$.
\begin{enumerate}[{\rm (i)}]
\item\label{natrualdef}
A curve $\gamma:I\rightarrow M$ is called {\it naturally absolutely continuous} if for any chart $(U;\varphi)$ of $M$, the composition
\[
\varphi\circ \gamma:\ \gamma^{-1}(\gamma(I)\cap U)\rightarrow \varphi(U)\subset \mathbb{R}^n
\]
is locally absolutely continuous, i.e., absolutely continuous on all bounded closed subintervals of $\gamma^{-1}(\gamma(I)\cap U)$.

\smallskip

\item\label{metriccontinuous} A curve $\gamma: I\rightarrow M$ is
called {\it metric absolutely continuous}, if for all $\varepsilon>0$ there is a $\delta> 0$ such that
\[
\sum_{i=1}^n d_F(\gamma(a_i), \gamma(b_i)) <\varepsilon,
\]
whenever $\{(a_i, b_i)\}^n_{i=1}$ are nonoverlapping subintervals of $I$ with
 $\sum_{i=1}^n (b_i-a_i) <\delta$.


\end{enumerate}

Let  $\NAC(I;M)$ (resp.,  $\MAC(I;M)$)   denote the class of naturally (resp., metric) absolutely continuous curves defined on $I$.

\end{definition}

Obviously, the naturally absolute  continuity is independent of the choice of Finsler structures. Moreover,
for any $\gamma\in \NAC(I;M)$,
 the derivative $\gamma'(t)$ exists for $\mathscr{L}^1$-a.e. $t\in I$. Thus, a standard argument (see  \cite[Theorem 2.2]{AB}) together with the proof of the Busemann-Mayer theorem (cf.  \cite[p.\,160]{BCS}) yields
\[
L_{d_F}(\gamma)=L_{F}(\gamma)=\int_I F(\gamma'(t)){\dd}t.
\]
If $I$ is a bounded closed interval, it is not hard to check $L_{F}(\gamma)<\infty$ and hence, $F(\gamma')\in L^1(I)$.
On the other hand,  metric absolutely continuous curves are the extension of  absolutely continuous functions  in classical real analysis. In view of Definition \ref{forwardabcontinc}, the first main result of this subsection reads as follows.

\begin{theorem}\label{newdefsabcon}
For every Finsler manifold $(M,F)$ and any  bounded closed interval $I$, there holds
\begin{equation}\label{abcurindentiy}
\NAC(I;M)=\MAC(I;M)=\AC(I;M)=\FAC(I;M)=\BAC(I;M).
\end{equation}
\end{theorem}

The assumption on $I$ is necessary. On the one hand, Example \ref{funcurve1} indicates that $I$ must be closed.  On the other hand, if $I$ is a unbounded closed interval, then  $\NAC(I;M)\neq\AC(I;M)$ even in the Euclidean setting (by considering $\gamma(t)=t$ in $(\mathbb{R},|\cdot|)$).

To prove  Theorem \ref{newdefsabcon}, we   establish a connection between $\NAC([0,1];M)$ and $\MAC([0,1];M)$ by the following result.
\begin{proposition}\label{dacconverge}
Let $(M,F)$ be a Finsler manifold and set
\begin{align*}
d_{\nac}(x_0,x_1)&:=\inf\left\{ L_F(\gamma)\,|\ \gamma\in \NAC([0,1];M),\ \gamma(0)=x_0, \ \gamma(1)=x_1  \right\},\quad \forall\,x_0,x_1\in M;\\
\mathcal {D}_{\nac}(\gamma_1,\gamma_2)&:=\sup_{t\in [0,1]}d_{\nac}(\gamma_1(t),\gamma_2(t))+\int_0^1 \left| F(\gamma_1'(t))-F(\gamma_2'(t))   \right|{\dd}t,\quad \forall\,\gamma_1,\gamma_2\in \NAC([0,1];M).
\end{align*}
Thus, $(M,d_{\nac})$ is an asymmetric metric space such that  its  forward topology coincides with the backward topology, which is exactly the original topology of $M$. Therefore, $\left( \NAC([0,1];M), \mathcal {D}_{\nac} \right)$ is an asymmetric metric space as well.
\end{proposition}
\begin{proof}
Owing to the connectivity of $M$, the length metric $d_{\nac}$ is always finite because every two points $p,q$ can be jointed by a piecewise smooth curve constructed by coordinate charts. Hence, $(M,d_{\nac})$ is an asymmetric metric space.

Now we claim that a sequence $(x_n)\subset M$ satisfies $d_{\nac}(x,x_n)\rightarrow 0$ if and only if  $d_F(x,x_n)\rightarrow 0$. Since $d_{\nac}\leq d_F$ follows from $\mathcal {A}_\infty([0,1];M)\subset\NAC([0,1];M)$, it suffices to show the ``only if" part. Provided $d_{\nac}(x,x_n)\rightarrow 0$,  the definition of $d_{\nac}$ yields a sequence $\gamma_n\in \NAC([0,1];M)$ from $x$ to $x_n$ such that  $d_{\nac}(x,x_n)\leq L_F(\gamma_n)\leq 2d_{\nac}(x,x_n)$, which furnishes $d_F(x,x_n)\leq L_{d_F}(\gamma_n)= L_F(\gamma_n)\rightarrow 0$.
So the claim is true, which implies that the  forward topology of $(M,d_{\nac})$ is the original topology of $M$.

Note that $\NAC([0,1];M)$ is independent of the choice of Finsler structures. Hence, by considering the reverse Finsler metric $\overleftarrow{F}(x,y):=F(x,-y)$, it follows by a similar argument to the above that  the  backward topology of $(M,d_{\nac})$ is also the original topology of $M$.

It remains to show the finiteness of $\mathcal {D}_{\nac}$ on $\NAC([0,1];M)$.
Given any $\gamma_1,\gamma_2\in \NAC([0,1];M)$, it is obvious that $\int_0^1 \left| F(\gamma_1'(t))-F(\gamma_2'(t))   \right|{\dd}t<\infty$. On the other hand, the finiteness of $\sup_{t\in [0,1]}d_{\nac}(\gamma_1(t),\gamma_2(t))$ follows from the compactness of $\gamma_1([0,1])\times \gamma_2([0,1])$   and the continuity of $d_{\nac}$ (with respect to  the original product  topology of $M\times M$), which concludes the proof.
\end{proof}

 Obviously, $d_{\nac}\leq d_F$ always holds. And if $(M,F)$ is either forward or backward complete, the standard theory of geodesic (cf. \cite{BCS,Sh1}) implies $d_{\nac}= d_F$. Now we show this identity remains true without the completeness assumption by the same method as employed in \cite{AB}.

\begin{theorem}\label{metricddsmo}For every Finsler manifold $(M,F)$,
 $\mathcal {A}_\infty([0,1];M)$ is dense in  $\left( \NAC([0,1];M), \mathcal {D}_{\nac} \right)$ with respect to the symmetrized topology and hence, $d_{\nac}=d_F$.
\end{theorem}
\begin{proof}  The proof is divided into two steps.

\smallskip

\textbf{Step 1.} Firstly, we show that $\mathcal {A}_\infty([0,1];M)$ is dense in  $\NAC([0,1];M)$.

\smallskip

Given $\gamma\in \NAC([0,1];M)$, choose a finite covering of chart $\{(U_\alpha,\varphi_\alpha)\}_{\alpha=1}^N$ such that $U_\alpha\Subset M$ and $\varphi_\alpha(U)$ is a convex open set of $\mathbb{R}^n$. Since $\cup_\alpha \overline{U}_\alpha$ is compact, there exists a constant $C>1$ such that for each $\alpha$,
\begin{equation}\label{normestimate11}
 C^{-1}\|\varphi_{\alpha*}y\|_\alpha\leq F(y)\leq C\|\varphi_{\alpha*}y\|_\alpha,\quad \forall\,y\in TU_\alpha,
\end{equation}
where $\varphi_{\alpha*}:TU_\alpha\rightarrow T\mathbb{R}^n$ is the tangent map and $\|\cdot\|_\alpha$ is the Euclidean norm on $\varphi_\alpha(U_\alpha)\subset \mathbb{R}^n$. Thus
\begin{equation} \label{distanceestimae}
C^{-1}\|\varphi_\alpha(x_1)-\varphi_\alpha(x_0)\|_\alpha\leq d_{\nac}(x_0,x_1)\leq C\|\varphi_\alpha(x_1)-\varphi_\alpha(x_0)\|_\alpha,\quad\forall\,x_0,x_1\in U_\alpha.
\end{equation}
Moreover, we choose a partition $0=t_0<t_1<\cdots<t_N=1$ of $[0,1]$ such that $\gamma|_{[t_{\alpha-1},t_\alpha]} \subset U_\alpha$ for all $\alpha$.

 For every $\alpha$, since $\varphi_\alpha\circ\gamma$ is  locally absolutely continuous, $\gamma$ is uniformly continuous over $[t_{\alpha-1},t_\alpha]$ and $F(\gamma') \in L^1([t_{\alpha-1},t_\alpha])$. Thus, given an arbitrary $\epsilon>0$,  there exists a small $\delta_\alpha\in (0,(t_\alpha-t_{\alpha-1})/2)$ such that
for any $t,T\in [t_{\alpha-1},t_\alpha]$ with $|T-t|<2\delta_\alpha$,
\[
\left\| \varphi_\alpha(\gamma(T))-\varphi_\alpha(\gamma(t)) \right\|_\alpha<\frac{\epsilon}{10CN},\quad \hat{d}_{\nac}(\gamma(t),\gamma(T))<\frac{\epsilon}{10CN},\quad \int^T_t F(\gamma'(s)){\dd}s<\frac{\epsilon}{10CN},
\]
where $\hat{d}_{\nac}$ is the symmetrized metric of ${d}_{\nac}$ (see \eqref{symmmetricde}).
By convolution with a mollifier $\rho$ we can get a componentwise regularization of $\varphi_\alpha\circ\gamma|_{[t_{\alpha-1},t_\alpha]}$. Thus,  there exists a small $\eta_\alpha>0$ such that the smooth approximation $\gamma_{\alpha}:=\varphi^{-1}_\alpha((\varphi_\alpha\circ \gamma)*\rho_{\eta_\alpha}) \in \mathcal {A}_\infty([t_{\alpha-1},t_\alpha];M)$ satisfies
\[
\sup_{t\in [t_{\alpha-1},t_\alpha]}\| \varphi_\alpha( \gamma(t))-\varphi_\alpha( \gamma_\alpha(t)) \|_\alpha<\frac{\epsilon}{10CN},\quad
\int^{t_\alpha}_{t_{\alpha-1}}\left\| (\varphi_\alpha\circ\gamma)'(s)- (\varphi_\alpha\circ\gamma_\alpha)'(s) \right\|_\alpha{\dd}s<\frac{\epsilon}{10CN},
\]
which together with (\ref{distanceestimae}), (\ref{normestimate11})   and the triangle inequality of $F$ furnishes
\[
\sup_{t\in [t_{\alpha-1},t_\alpha]} \hat{d}_{\nac}(\gamma(t),\gamma_\alpha(t))<\frac{\epsilon}{10N},\quad \int^{t_\alpha}_{t_{\alpha-1}}\left| F(\gamma'(s))- F(\gamma_\alpha'(s)) \right|{\dd}s<\frac{\epsilon}{10N}.
\]

Note that the  concatenation   $\gamma_1*\cdots*\gamma_\alpha*\cdots*\gamma_N$ is usually discontinuous. In the sequel we construct a continuous and piecewise smooth curve to approximate $\gamma$.
Let $\nu_{\alpha-1}:[t_{\alpha-1},t_{\alpha-1}+\delta_\alpha]\rightarrow U_\alpha$ denote a smooth curve from $\gamma(t_{\alpha-1})$ to $\gamma_{\alpha}(t_{\alpha-1}+\delta_\alpha)$ such that $\varphi_\alpha(\nu_{\alpha-1})$ is a straight line. The convexity of $\varphi_\alpha(U_\alpha)$ implies that  $\nu_{\alpha-1}$ is well defined.  Similarly define a smooth curve $\mu_\alpha:[t_\alpha-\delta_\alpha,t_\alpha]\rightarrow U_\alpha$ from $\gamma_\alpha(t_\alpha-\delta_\alpha)$ to $\gamma(t_\alpha)$. Now we define a curve $\zeta_\alpha:[t_{\alpha-1},t_\alpha]\rightarrow U_\alpha$ by $\zeta_\alpha:=\nu_{\alpha-1}*\gamma_\alpha|_{[t_{\alpha-1}+\delta_\alpha,t_\alpha-\delta_\alpha]}*\mu_\alpha$, which is a piecewise smooth curve from $\gamma(t_{\alpha-1})$ to $\gamma(t_\alpha)$. A direct but tedious calculation similar to that in \cite[p.\,283]{AB}  yields
\[
\sup_{t\in [t_{\alpha-1},t_\alpha]}\hat{d}_{\nac}(\gamma(t),\zeta_\alpha(t))+\int^{t_\alpha}_{t_{\alpha-1}}\left| F(\gamma')-F(\zeta'_\alpha) \right|{\dd}s
\leq \frac{3\epsilon}{10N}+\frac{7\epsilon}{10N}= \frac{\epsilon}{N}.
\]
Hence, the concatenation  $\zeta_\epsilon(t):=\zeta_1*\cdots *\zeta_N$ is a piecewise smooth curve from $\gamma(0)$ to $\gamma(1)$ such that $\hat{\mathcal {D}}_{\nac}(\gamma,\zeta_\epsilon)<\epsilon$, where $\hat{\mathcal {D}}_{\nac}$ is the symmetrized metric of ${\mathcal {D}}_{\nac}$. Therefore,  the density of $\mathcal {A}_\infty([0,1];M)$ in  $\NAC([0,1];M)$ follows.

\medskip

\textbf{Step 2.} Secondly, we show $d_{\nac}=d_F$.

\smallskip

Since $d_{\nac}\leq d_F$, it suffices to show the reverse inequality. Given any $x_0,x_1\in M$ and any $\epsilon>0$, there exists $\gamma\in \NAC([0,1];M)$ such that $\gamma(0)=x_0$, $\gamma(1)=x_1$ and
 $L_F(\gamma)<d_{\nac}(x_0,x_1)+\epsilon$. The density of $\mathcal {A}_\infty([0,1];M)$ yields a sequence $(\gamma_n)\subset \mathcal {A}_\infty([0,1];M)$ with $\mathcal {D}_{\nac}(\gamma,\gamma_n)\rightarrow 0$. Hence, there is $N=N(\epsilon)>0$ such that
 \begin{equation}\label{bacdefconc}
 d_F(\gamma_n(0),\gamma_n(1))\leq L_F(\gamma_n)\leq L_F(\gamma)+\epsilon<d_{\nac}(x_0,x_1)+2\epsilon,\quad\forall\, n>N.
 \end{equation}
 Since $\mathcal {D}_{\nac}(\gamma,\gamma_n)\rightarrow 0$ implies $  d_{\nac}(\gamma(0),\gamma_n(0))\rightarrow 0$ and $d_{\nac}(\gamma(1),\gamma_n(1))\rightarrow0$, Proposition \ref{dacconverge} furnishes
  \[
  \lim_{n\rightarrow \infty}d_F(x_0,\gamma_n(0))=0=\lim_{n\rightarrow \infty}d_F(x_1,\gamma_n(1)),
  \] which together with \eqref{bacdefconc} yields $d_F\leq d_{\nac}$.
\end{proof}

Now we proceed to prove Theorem \ref{newdefsabcon}.


\begin{proof}[Proof of Theorem \ref{newdefsabcon}] For convenience, we assume $I=[0,1]$. We first show $\FAC([0,1];M)\subseteq \MAC([0,1];M)$.
Given $\gamma\in \FAC([0,1];M)$, there exists a nonnegative $f\in L^1([0,1])$ such that for any $[t_1,t_2]\subset [0,1]$,
\[
d_F(\gamma(t_1),\gamma(t_2))\leq \int^{t_2}_{t_1} f(t){\dd}t=G(t_2)-G(t_1),
\]
where
  $G(s):=\int_0^s f(t){\dd}t$.
Thus
$\gamma\in \MAC([0,1];M)$ since $G(s)$ is an absolutely continuous function.

 Secondly, we show
$\MAC([0,1];M)\subseteq\NAC([0,1];M)$.
Given $\gamma\in \MAC([0,1];M)$, consider a chart $(U,\varphi)$ with $U\cap \gamma([0,1])\neq\emptyset$. Choose any $[a,b]\subset [0,1]$ with $\gamma([a,b])\subset U$. Without loss of generality,  we may assume that $\overline{U}$ is compact (otherwise choose a smaller open set $U'\subset U$ such that $\gamma([a,b])\subset U'$ and $\overline{U'}$ is compact). The same argument as in the proof of Theorem \ref{metricddsmo} yields  a constant $C>0$ such that
\[
\|\varphi(\gamma(t_2))-\varphi(\gamma(t_1))\|\leq   C d_F(\gamma(t_1),\gamma(t_2)),\quad a\leq t_1\leq t_2\leq b,
\]
where $\|\cdot\|$ is the Euclidean norm in $\varphi(U)\subset \mathbb{R}^n$. Since $\gamma$ is metric absolutely continuous, $\varphi\circ \gamma$ is absolutely continuous over $[a,b]$, that is $\gamma\in \NAC([0,1];M)$.

Thirdly, we prove $\NAC([0,1];M)\subseteq \FAC([0,1];M)$.
For any $\gamma\in \NAC([0,1];M)$, the derivative $\gamma'$ exists for $\mathscr{L}^1$-a.e. $t\in (0,1)$  and $F(\gamma')\in L^1([0,1])$. Thus, Theorem \ref{metricddsmo} yields
\[
d_F(\gamma(t_1),\gamma(t_2))=d_{\nac}(\gamma(t_1),\gamma(t_2))\leq \int^{t_2}_{t_1}F(\gamma'){\dd}t, \quad 0\leq t_1\leq t_2\leq 1,
\]
which implies $\gamma\in \FAC([0,1];M)$.

From above, we obtain $\NAC([0,1];M)=\MAC[0,1];M)=\FAC([0,1];M)$. It remains to show $$\FAC([0,1];M)=\BAC([0,1];M)=\AC([0,1];M).$$ Note that $\NAC([0,1];M)$ is independent of the choice of the Finsler metric. Hence,  considering the reverse Finsler manifold $(M,\overleftarrow{F})$, we have $\NAC([0,1];M)=\BAC([0,1];M)$, which concludes the proof.
 \end{proof}

\begin{corollary}\label{derivativeabsolu}Let $(M,F)$ be a generic Finsler manifold and let $I$ be an interval of $\mathbb{R}$. Thus,
for any $\gamma\in \FAC^p(I;M)$ (resp., $\gamma\in \BAC^p(I;M)$) with $p\in [1,\infty]$, there holds
\begin{equation}\label{dervatexsfin}
|\gamma'_+|(t)=F(\gamma'(t)) \quad (\text{resp., }|\gamma'_-|(t)=F(-\gamma'(t))),\quad \text{for $\mathscr{L}^1$-a.e. $t\in \intt(I)$}.
\end{equation}
\end{corollary}
\begin{proof}We only prove the case of $\gamma\in \FAC^p(I;M)$. The other case can be deduced by considering the reverse Finsler manifold.
Given $\gamma \in \FAC^p(I;M)$, let $J\subset I$ be a bounded closed interval.
Thus,   an easy argument together with the H\"older inequality yields $\gamma|_J\in \FAC^p(J;M)\subset \FAC(J;M)$, which together with Theorem \ref{newdefsabcon} implies that $\gamma$ is differentiable for $\mathscr{L}^1$-a.e. $t\in \intt(J)$. Thus, the proof of the Busemann-Mayer theorem (cf. \cite[p.\,160]{BCS}) together with  Proposition \ref{baspeedforback} yields
\[
F(\gamma'(t))=\lim_{h\rightarrow0^+}\frac{d_F(\gamma(t),\gamma(t+h))}{h}=|\gamma'_+|(t),\quad \text{for $\mathscr{L}^1$-a.e. $t\in \intt(J)$}.
\]
The corollary follows by a countable partition $(J_i)$ of $I$  such that each $J_i$ is a bounded closed interval.
\end{proof}

The following result is a partially stronger version of Theorem \ref{metricddsmo}.
\begin{theorem}\label{FACBAC}
For every Finsler manifold $(M,F)$ and any bounded closed interval $I$,
\begin{equation*}\label{pcaseidentity}
\FAC^p(I;M)=\BAC^p(I;M)=\AC^p(I;M),\quad \forall\,p\in [1,\infty].
\end{equation*}
\end{theorem}
\begin{proof}It is enough to show $\FAC^p(I;M)=\BAC^p(I;M)$. Without loss of generality, we assume $I=[0,1]$.
Given $\gamma \in \FAC^p([0,1];M)$, it follows from  Corollary \ref{derivativeabsolu}, Proposition \ref{baspeedforback} and the H\"older inequality that $F(\gamma'(t))\in L^p([0,1])\subset L^1([0,1])$.
The compactness of  $\gamma([0,1])$ implies  the reversibility $\theta:=\lambda_F(\gamma([0,1]))<\infty$. Thus, for any $[t_1,t_2]\subset [0,1]$, we have
\begin{equation}\label{reernoncurve}
d_F(\gamma(t_2),\gamma(t_1))\leq
\int_{1-t_2}^{1-t_1}F(-\gamma'(1-t)){\dd}t
=\int_{t_1}^{t_2}F(-\gamma'(t)){\dd}t
\leq\theta\int_{t_1}^{t_2}F(\gamma'(t)){\dd}t,
\end{equation}
which implies $\gamma\in \BAC^p([0,1];M)$ and therefore,
$\FAC^p([0,1];M)\subseteq \BAC^p([0,1];M)$. The reverse relation follows by considering the reverse Finsler manifold.
\end{proof}
\begin{remark}\label{closureisbounded}
The assumption about $I$ in Theorem \ref{FACBAC} is also necessary. In fact, Example \ref{funcurve1} indicates that $I$ should be closed. Moreover,  let $\xi:[0,\infty)\rightarrow \mathbb{B}^n$ be the  unit speed minimizing geodesic from $\mathbf{0}$ to $(-1,0,\ldots,0)$ in the Funk space. A direct calculation yields $\xi \in \BAC([0,\infty);\mathbb{B}^n)\backslash \FAC([0,\infty);\mathbb{B}^n)$, which implies the necessity of the boundedness of $I$.
\end{remark}

The following result is an immediate consequence of Theorem \ref{FACBAC}. 
\begin{corollary}\label{abcurcoin}
$\AC^p([0,1];M)$ is  the class of $p$-absolutely continuous curves defined on $[0,1]$ in $(M,\hat{d}_F)$.
\end{corollary}

In the spirit of \cite[Lemma 2.1]{Li2}, we obtain the following result, which plays an important role in the proof of Theorem \ref{maintheorem}.

\begin{theorem}\label{strongcontrall}Let $(M,F)$ be a forward or backward complete Finsler manifold.
Suppose that $\gamma:[0,1]\rightarrow (M,d_F)$ is a map satisfying
\begin{itemize}

\item $\gamma$ is right-continuous at every $t\in [0,1]$ and continuous except a countable set in $[0,1]$;

\smallskip

\item there is a bounded nondecreasing function $v:[0,1]\rightarrow [0,\infty)$ such that
\begin{equation}\label{contbounded}
d_F(\gamma(s),\gamma(t))\leq v(t)-v(s),\quad \forall [s,t]\subset [0,1];
\end{equation}

\item by extending $\gamma(1+\varepsilon):=\gamma(1)$ for any $\varepsilon>0$, the following limit is bounded
\begin{equation}\label{normcontrall}
\limsup_{h\rightarrow 0^+}\left\|  \frac{d_F(\gamma(\cdot),\gamma(\cdot+h))}{h}   \right\|_{L^p([0,1])}<\infty.
\end{equation}

\end{itemize}
Thus, $\gamma\in \AC^p([0,1];M)$.
\end{theorem}

\begin{proof} By assumption, there exists  $C\in [0,\infty)$ such that $v(t)\leq C$ for all $t\in [0,1]$. Thus, \eqref{contbounded} yields
\[
\overline{\gamma([0,1])}\subset \overline{B^+_{\gamma(0)}(C)}\cap \overline{B^-_{\gamma(1)}(C)}=:\mathscr{B}.
\]
The completeness condition of $(M,F)$ implies that $\mathscr{B}$ is a compact set. Hence,
 the compactness as well as
the separability of $\overline{\gamma([0,1])}$ follow. So we can choose a countable dense subset $(x_n)$ of $\gamma([0,1])$. For every fixed $n\in \mathbb{N}$, set $\gamma_n(t):=d_F(x_n,\gamma(t))$. The triangle inequality of $d_F$ yields
\[
\gamma_n(t+h)-\gamma_n(t)\leq d_F(\gamma(t),\gamma(t+h)),
\]
which combined with $\theta:=\lambda_F\left(\mathscr{B}\right)<\infty$ furnishes
\begin{equation}\label{derviavetcontroll}
\left|\gamma_n(t+h)-\gamma_n(t)\right|\leq \theta d_F(\gamma(t),\gamma(t+h)), \quad \forall\,t\in [0,1],\ h>0.
\end{equation}

Now $\gamma_n\in W^{1,1}([0,1])$ follows from a standard argument (cf.  \cite[p.674]{Li2} or \cite[p.29]{AGS}) and \eqref{normcontrall}. Thus, \eqref{derviavetcontroll} yields for $\mathscr{L}^1$-a.e. $t\in [0,1]$,
\begin{equation}\label{absdervative}
|\gamma_n'|(t):=\lim_{h\rightarrow 0^+}\frac{\left|\gamma_n(t+h)-\gamma_n(t)\right|}{h}\leq \theta \liminf_{h\rightarrow 0^+}\frac{d_F(\gamma(t),\gamma(t+h))}{h}.
\end{equation}
Let
\[
\mathscr{N}:=\bigcup_{n\in \mathbb{N}}\left\{ t\in [0,1]\,|\, \text{$\gamma_n'(t)$ does not exist} \right\},\qquad m(t):=\sup_{n\in \mathbb{N}}|\gamma'_n|(t) \ \text{ for }t\in [0,1]\backslash \mathscr{N}.
\]
Since $\mathscr{N}$ is a Lebesgue null set, Fatou's lemma together with \eqref{normcontrall} yields
\begin{align*}
\|m\|_{L^p([0,1])}\leq \theta  \liminf_{h\rightarrow 0^+}\left\|  \frac{d_F(\gamma(\cdot),\gamma(\cdot+h))}{h}   \right\|_{L^p([0,1])}\leq \theta \limsup_{h\rightarrow 0^+}\left\|  \frac{d_F(\gamma(\cdot),\gamma(\cdot+h))}{h}   \right\|_{L^p([0,1])}<\infty,
\end{align*}
i.e., $m\in L^p([0,1])$. Moreover,
 for any $[s,t]\subset [0,1]$, the density of $(x_n)$ yields
\begin{equation*}
d_F(\gamma(s),\gamma(t))=\sup_{n\in \mathbb{N}}\left(\gamma_n(t)-\gamma_n(s)\right)\leq \sup_{n\in \mathbb{N}}\left|\gamma_n(t)-\gamma_n(s)\right|\leq \sup_{n\in \mathbb{N}}\int^t_s |\gamma'_n|(r){\dd}r\leq \int^t_s m(r){\dd}r,
\end{equation*}
which combined with Theorem \ref{FACBAC} concludes the proof.
\end{proof}


The following result serves as a basic tool to introduce the gradient flow  to the Finsler setting (see Section \ref{grflofnm} below).
\begin{proposition}\label{smoothisabsoltey}
Let $(M,F)$ be a  forward complete Finsler manifold and let $\phi:M\rightarrow \mathbb{R}$ be a $C^1$-function. Thus,  $\phi\circ\gamma:[0,1]\rightarrow \mathbb{R}$ is absolutely continuous for any $\gamma\in \AC([0,1];M)$ and hence,
\begin{equation}\label{basicestimategamaphi}
\frac{\dd}{{\dd}t}\left(\phi\circ\gamma\right)=\langle \gamma'(t),{\dd}\phi\rangle\geq -F(\gamma'(t))F(\nabla(-\phi)|_{\gamma(t)}), \quad \text{for $\mathscr{L}^1$-a.e. $t\in (0,1)$}.
\end{equation}
\end{proposition}
\begin{proof} Given a curve $\gamma\in \AC([0,1];M)$, set  $\theta:=\lambda_F\left( \gamma([0,1]) \right)<\infty$ and  $R:=(\theta +1)L_F(\gamma)+1<\infty$. Thus, $\gamma([0,1])$ is contained in $B^+_{\gamma(0)}(R)$.
The forward completeness implies the precompactness of $B^+_{\gamma(0)}(R)$ and hence,
\[
C:=\sup\left\{ \lambda_F(x) F^*({\dd}\phi|_x)\,\Big|\ x\in  \overline{B^+_{\gamma(0)}(R)} \right\}<\infty.
\]
Given any $a,b\in [0,1]$, there exists a unit speed minimizing  geodesic $\zeta(s)$, $s\in [0,d_F(\gamma(a),\gamma(b))]$ from $\gamma(a)$ to $\gamma(b)$ (cf. \cite[Proposition 6.5.1]{BCS}). The triangle inequality of $d_F$ implies that
 $\zeta$ is contained in $B^+_{\gamma(0)}(R)$.
 By (\ref{tangninnerprot}) we get
\[
\left|(\phi\circ\zeta)'(s)\right|=\left|\langle \zeta'(s),{\dd}\phi\rangle\right|\leq \lambda_F(\zeta(s)) F(\zeta'(s))F^*({\dd}\phi)=\lambda_F(\zeta(s))F^*({\dd}\phi|_{\zeta(s)})\leq C,
\]
which together with the mean value theorem yields
\begin{equation}\label{abcontofphi}
|\phi\circ\gamma(b)-\phi\circ\gamma(a)|=|\phi\circ\zeta(d_F(\gamma(a),\gamma(b)))-\phi\circ\zeta(0)|\leq C\,d_F(\gamma(a),\gamma(b)).
\end{equation}
Now the absolute  continuity of $\phi\circ\gamma$   is a direct consequence of \eqref{abcontofphi} and \eqref{moredfabsc}. Hence, the derivative of  $\phi\circ\gamma$ exists for $\mathscr{L}^1$-a.e. $t\in (0,1)$ and therefore, \eqref{basicestimategamaphi} follows from \eqref{tangninnerprot} directly.
 \end{proof}


\subsection{Gradient flows in Finsler manifolds}\label{grflofnm}

In this subsection, we discuss briefly the theory of gradient flow in the Finsler setting. Also refer to \cite{CRZ,RMS,OZ} for the results in the context of  general asymmetric metric spaces.

\begin{definition}
Let $(M,F)$ be a Finsler manifold and let $I$ be an interval of $\mathbb{R}$. A curve $\gamma:I\rightarrow M$ is called {\it locally absolutely continuous} (denoted by $\gamma\in \AC_{\lo}(J;M)$) if the restriction $\gamma|_J\in \AC(J;M)$ for every bounded closed interval $J\subset I$.
\end{definition}
\begin{remark}
Although one can similarly define locally forward/backward absolutely continuous curves defined on $I$, these classes are exactly  $\AC_{\lo}(J;M)$ due to  Theorem \ref{FACBAC}.
\end{remark}

Inspired by \cite{AGS,RMS}, we introduce the following notations.
Let $h:[0,\infty)\rightarrow [0,\infty)$ be  a continuous, subjective and strictly increasing function. The corresponding convex primitive function $\psi$ and the Legendre-Fenchel-Moreau transform $\psi^*$ are defined by
\[
\psi(x):=\int^x_0 h(r){\dd}r,\quad \psi^*(y)=\sup_{x\geq 0}\left[ xy-\psi(x)  \right].
\]
In particular, for any nonnegative   numbers $x,y$, there holds
\begin{equation}\label{equlfm}
xy\leq \psi(x)+\psi^*(y), \text{ with equality if and only if } y=h(x)=\psi'(x).
\end{equation}
\begin{definition}
Let $(M,F)$ be a forward complete Finsler manifold, $\phi:M \rightarrow \mathbb{R}$ be a $C^1$-function and $I$ be an interval of $\mathbb{R}$.
A curve $\xi \in \AC_{\loc}(I;M)$ is called
a \emph{generalized gradient flow} for $\phi$ with respect to $\psi$
if
\begin{equation}\label{curvemaxforslop}
 \frac{\dd}{{\dd}t}(\phi\circ\xi)(t) = -\psi(F(\xi'(t)))-\psi^*{\left(F(\nabla(-\phi)|_{\xi(t)}\right)}
 \quad \text{for $\mathscr{L}^1$-a.e.\ $t\in \intt(I)$}.
\end{equation}
\end{definition}
  Equation \eqref{curvemaxforslop} is well defined due to Proposition \ref{smoothisabsoltey} and it is
 an extension of standard gradient flows.
\begin{example}[\cite{OZ}]
Let $\psi(x):=x^p/p$ and $p\in (1,\infty)$. Thus, \eqref{curvemaxforslop} becomes
\[
\frac{\dd}{{\dd}t}(\phi\circ\xi)(t)= -\frac1p F^p(\xi'(t))-\frac1q F^q\left(\nabla(-\phi)|_{\xi(t)}\right),\quad \text{for $\mathscr{L}^1$-a.e.\ $t\in \intt(I)$},
\]
where $q$ is the conjugate exponent of $p$, i.e., $1/p+1/q=1$. An easy argument together with \eqref{basicestimategamaphi} and the Young inequality yields
$\mathfrak{j}_p(\xi'(t))=\nabla(-\phi)(\xi(t))$,
where $\mathfrak{j}_p:TM \rightarrow  TM$ is defined by
$\mathfrak{j}_p(v):=F^{p-2}(v)v$ if $v\neq0$ and $\mathfrak{j}_p(0):=0$.
When $p=2$, we obtain the usual gradient flow equation $\xi'(t)=\nabla(-\phi)(\xi(t))$.
In particular, $\xi'(t)=-\nabla\phi(\xi(t))$ holds only if $F$ is reversible.
\end{example}

The existence and regularity of solutions to Equation \eqref{curvemaxforslop} reads as follows.
\begin{theorem}\label{existenceofgenreaflow}
Let $(M,F)$ be a forward complete Finsler manifold and let $\phi:M \rightarrow \mathbb{R}$ be a $C^1$-function. Thus, for any $x_0 \in M$,
there exists a $C^1$-curve $\xi:[0,T) \rightarrow M$ solving the generalized gradient flow equation \eqref{curvemaxforslop} with $\xi(0)=x_0$, in which case there holds the energy identity
 \begin{equation}\label{energyidenty}
 (\phi\circ\xi)(s)-(\phi\circ\xi)(t)=\int^t_s\psi(F'(\xi(r))){\dd}r+\int^t_s\psi^*{\left(F(\nabla(-\phi)|_{\xi(r)})\right)}{\dd}r,\quad \forall\,s,t\in [0,T).
 \end{equation}
 Here, $T$ denotes the  maximal existence time, which satisfies
 \begin{itemize}
 \item [(i)] if  $T<\infty$, then $\lim_{t \to T} d_F(x_0,\xi(t))=\infty$;

 \smallskip

 \item [(ii)] if $\phi$ is bounded, then the maximal time $T=\infty$.

 \end{itemize}
 \end{theorem}
\begin{proof}
We divide the proof into two steps.

\smallskip

 \textbf{Step 1.} Suppose that $\phi$ is bounded.  Then the existence of $\xi\in \AC_{\loc}([0,\infty);M)$ to Equation  \eqref{curvemaxforslop} with $\xi(0)=x_0$ follows from \cite[Theorem 3.5]{RMS}, \cite[Example 2.20]{OZ} and Proposition \ref{smoothisabsoltey} directly  (in \cite[Theorem 3.5]{RMS} choosing $\Delta:=d_F$, $X_0:=M$, $\mathcal {E}_t:=\phi$  and $\sigma:=$ the original topology of $M$ and noting $|\partial^-\mathcal {E}_t|=|\partial \mathcal {E}_t|=F(\nabla(-\phi))$). Furthermore, \eqref{equlfm} together with \eqref{curvemaxforslop} and \eqref{basicestimategamaphi} yields
\begin{equation}\label{derenergide}
 \frac{\dd}{{\dd}t}(\phi\circ\xi)(t)=-F(\xi'(t))\,F(\nabla(-\phi)|_{\xi(t)})= -\psi(F(\xi'(t)))-\psi^*{\left(F(\nabla(-\phi)|_{\xi(t)}\right)},\quad \forall\,t\in (0,\infty),
 \end{equation}
which implies \eqref{energyidenty}. Besides, the left-hand side of \eqref{derenergide} combined with \eqref{basicestimategamaphi} and \eqref{tangninnerprot} yields  a nonnegative function $\alpha(t)$ such that $\xi'(t)=\alpha(t)\nabla(-\phi)|_{\xi(t)}$. Moreover, since $h(0)=0$ and $h^{-1}$ exists,  the right-hand side of \eqref{derenergide} together with \eqref{equlfm} furnishes
\begin{equation}\label{finserpflow}
 \xi'(t)=\left\{
\begin{array}{lll}
 h^{-1}{\left( F(\nabla(-\phi)|_{\xi(t)})  \right)} \frac{\nabla(-\phi)|_{\xi(t)} }{F(\nabla(-\phi)|_{\xi(t)})}, &
 \text{ if } \nabla(-\phi)\big( \xi(t) \big) \neq0,\\
\\
 0, & \text{ if } \nabla(-\phi) \big( \xi(t) \big)=0.
\end{array}
\right.
\end{equation}
Therefore, $ \xi$ is $C^1$ since $\phi\in C^1(M)$ and $h$ is continuous as well as increasing.

\smallskip

\textbf{Step 2.} Suppose that $\phi$ is unbounded. We replace $\phi$ with $\phi_r:=\phi|_{B^+_{x_0}(r+1)}$ for large $r>0$. The forward completeness implies that $B^+_{x_0}(r+1)$ is precompact and hence, $\phi_r$ is bounded.
By Step 1, we can construct a gradient curve $\xi$ within $B^+_{x_0}(r)$.
If $\xi$ does not reach $\partial B^+_{x_0}(r)$, then $\xi$ is defined on $[0,\infty)$.
If $\xi$ reaches $\partial B^+_{x_0}(r)$ at some $T_1 \in (0,\infty)$,
then we continue the construction for $\phi_{2r}$ from $x_1:=\xi(T_1)$.
Iterating this procedure, since $\nabla(-\phi_r)=\nabla(-\phi)$ in $B^+_{x_0}(r)$,
we eventually obtain a $C^1$-curve $\xi:[0,T) \rightarrow M$
satisfying \eqref{energyidenty} and $\xi(0)=x_0$. In particular, the construction implies
$\lim_{t \to T} d_F(x_0,\xi(t))=\infty$ if $T<\infty$.
 \end{proof}

\begin{remark}\label{regularifinslerm}
For $\phi \in C^l(M)$ with $l \ge 2$,
$\nabla(-\phi)$ is only continuous at its zeros while $C^{l-1}$ at other points
(see \cite{GS,OS}).
Hence, in order to improve the regularity of $\xi$, by \eqref{finserpflow}
we need  addition conditions to make sure  $\nabla(-\phi)(\xi(t)) \neq 0$.
See \cite[Corollary 4.16]{OZ} for instance.
\end{remark}

\section{Absolutely continuous curves in Wasserstein spaces over Finsler manifolds }\label{asbwassficcns}


We begin this section by recalling some basic concepts in the measure theory (cf. \cite{Bo1,Bo2}).
Let $X,Y$ be two measurable spaces. Denote by $\mathscr{P}(X),\mathscr{P}(Y)$   the collections of Borel probability measures on $X,Y$, respectively.
 \begin{itemize}
 \item[(a)] A sequence $(\mu_n)\subset \mathscr{P}(X)$ is said to {\it narrowly convergent to} $\mu\in \mathscr{P}(X)$ as $n\rightarrow \infty$ if
 \begin{equation}\label{narrwconver}
 \lim_{n\rightarrow \infty}\int_X f(x){\dd}\mu_n(x)=\int_X f(x){\dd}\mu(x),\quad \forall\,f\in C_b(X),
 \end{equation}
 where $C_b(X)$ is the space of continuous and bounded real functions defined on $X$. For convenience, we use ``$\mu_n\Rightarrow \mu$" to denote the narrow convergence.

 \smallskip

 \item[(b)] Given a map $h:X\rightarrow Y$ and a measure $\mu\in \mathscr{P}(X)$, the {\it push-forward measure} $h_\sharp \mu\in \mathscr{P}(Y)$ is defined as $h_\sharp \mu:=\mu\circ h^{-1}$. For any measurable function $f:Y\rightarrow [-\infty,\infty]$, there holds
 \[
 \int_Y f(y){\dd}h_\sharp\mu(y)=\int_X f\circ h(x) {\dd}\mu(x).
 \]
 In particular, $h_\sharp:\mathscr{P}(X)\rightarrow \mathscr{P}(Y)$ is continuous if $h$ is continuous.
 \end{itemize}

Now we recall Prokhorov's theorem  (cf. \cite[Theorem 5.1.3]{AGS}).
\begin{theorem}\label{grenarlizedpROKTHE}
Let $(X,\hat{d})$ be a   Polish space, i.e., a separable complete symmetric metric space. A set $\mathscr{T}\subset \mathscr{P}(X)$ is precompact in $\mathscr{P}(X)$ (with respect to the narrow  convergence) if and only if $\mathscr{T}$ is tight, i.e.,
\begin{equation}\label{tighnessporos}
\forall\varepsilon>0,\quad \exists K_\varepsilon\subset X \text{ is compact such that }\mu(X\backslash K_\varepsilon)<\varepsilon,\quad \forall\, \mu\in\mathscr{T}.
\end{equation}
\end{theorem}
\begin{remark}\label{anoterhtightorp}Owing to \cite[Remark 5.1.5]{AGS},
the tightness (\ref{tighnessporos}) is equivalent to the following condition: there exists a function $\varphi:X\rightarrow [0,\infty]$ such that
\begin{itemize}

\item for any $c\geq 0$, the sublevel $\lambda_c(\varphi):=\{x\in X\,|\ \varphi(x)\leq c \}$ is compact in $X$;

\smallskip

\item $\ds \sup_{\mu\in \mathscr{T}}\int_{X}\varphi(x){\dd}\mu(x)<\infty$.

\end{itemize}
\end{remark}
\subsection{Measurable curves in Finsler manifolds}\label{measfinsler}
In the sequel, we always assume that $(M,F)$ is a forward complete Finsler manifold. We use $d_F$  and $\hat{d}_F$   to denote respectively the distance function induced by $F$ and the symmetrized metric of $d_F$ defined as in \eqref{symmmetricde}. Recall that the forward topology of $(M,d_F)$ coincides with the backward topology, which is exactly the original topology of $M$. Thus, by Theorem \ref{asymmetrtopol} we have
\begin{proposition}\label{completedifin}Let $(M,F)$ be a forward complete Finsler manifold.
The space $(M,\hat{d}_F)$ is a Polish space, whose metric topology is exactly the original topology of $M$.
\end{proposition}

In the spirit of \cite[Section 2.4]{Li2}, we introduce a bounded symmetric  metric $\mathfrak{d}_F$ on $M$ by
  \[
  \mathfrak{d}_F(x_1,x_2):=\min\left\{\hat{d}_F(x_1,x_2),1\right\}.
  \]
Owing to Proposition \ref{completedifin}, the topology induced by $\mathfrak{d}_F$ is exactly the original topology of $M$.

Let $\mathcal {M}([0,1];M)$ be the collection of Lebesgue equivalent classes of Lebesgue measurable map from $[0,1]$ to $M$, i.e., for any $[\gamma]\in \mathcal {M}([0,1];M)$,
\[
[\gamma]:=\left\{\zeta:[0,1]\rightarrow M\,\Big|\, \gamma \text{ is Lebesgue measurable and $\zeta(t)=\gamma(t)$ for $\mathscr{L}^1$-a.e. $t\in [0,1]$}  \right\}.
\]
For convenience, we will use $\gamma$ to denote $[\gamma]$ in the sequel.
Equip $\mathcal {M}([0,1];M)$ with the symmetric metric $$\mathfrak{d}_1(
\gamma_1,\gamma_2):=\int^1_0\mathfrak{d}_F(\gamma_1,\gamma_2){\dd}t.$$
Thus, $(\mathcal {M}([0,1];M),\mathfrak{d}_1)$ is a Polish space.
Moreover, it follows from Chebyshev's inequality that  if a sequence $(\gamma_n)$ converges to $\gamma$ in    $\mathcal {M}([0,1];M)$, then $(\gamma_n)$ converges to $\gamma$ in measure, i.e.,
\begin{equation}\label{conveginmeasure}
\lim_{n\rightarrow \infty}\mathscr{L}^1\left(\left\{t\in [0,1]\,\Big|\, \hat{d}_F(\gamma(t),\gamma_n(t)) >\varepsilon  \right\}      \right)=0,\quad \forall\,\varepsilon>0.
\end{equation}

  Since the  structure of $\mathcal {M}([0,1];M)$ is factually established over the symmetric metric space $(M,\hat{d}_F)$, the following result is a direct consequence of  \cite[Theorem 2.2]{Li2} (also see \cite[Theorem 2]{RS}) combined with Proposition \ref{completedifin}.
\begin{theorem}\label{tightcompreversible}
Let $(M,F)$ be a forward complete Finsler manifold and let $p\in [1,\infty)$. A family $\mathscr{A}\subset \mathcal {M}([0,1];M)$ is precompact (with respect to $\mathfrak{d}_1$) if
\begin{itemize}

\item[(1)] $\ds\lim_{h\rightarrow 0^+}\sup_{\gamma\in \mathscr{A}}\int^{1-h}_0 d_F(\gamma(t),\gamma(t+h)){\dd}t=0  $;

\item[(2)] there exists a function $\psi:M\rightarrow [0,\infty]$ such that $\ds\sup_{\gamma\in \mathscr{A}}\int^1_0 \psi(\gamma(t)){\dd}t<\infty$ and  the sublevel $\lambda_c(\psi):=\{x\in M\,|\ \psi(x)\leq c\}$ is compact for every $c\geq 0$.

\end{itemize}
\end{theorem}

 We also recall the following result (cf. \cite[Lemma 2.3]{Li2}).
\begin{lemma}\label{mcontress}
Suppose that a sequence $(\widetilde{\eta}_n)\subset \mathscr{P}(\mathcal {M}([0,1];M))$ satisfies the following conditions:
\begin{itemize}
\item $(\widetilde{\eta}_n)$ narrowly converges to $\widetilde{\eta}\in \mathscr{P}(\mathcal {M}([0,1];M))$;

\item there is a sequence of $\widetilde{\eta}_n$-measurable functions $L_n:\mathcal {M}([0,1];M)\rightarrow[0,\infty)$ with
 \begin{equation*}
 \sup_{n\in \mathbb{N}}\int_{\mathcal {M}([0,1];M)} L_n(\gamma){\dd}\widetilde{\eta}_n(\gamma)<\infty.
 \end{equation*}

\end{itemize}
Thus, there is a subsequence $(\widetilde{\eta}_{n_k})$ such that for $\widetilde{\eta}$-a.e. $\gamma\in \supp\widetilde{\eta}$, there is a $\gamma_{n_k}\in \supp\widetilde{\eta}_{n_k}$ with
\begin{equation*}\label{d1suplboundedness}
\lim_{k\rightarrow \infty}\mathfrak{d}_1(\gamma,\gamma_{n_k})=0,\quad \sup_{k\in \mathbb{N}}L_{n_k}(\gamma_{n_k})<\infty.
\end{equation*}
\end{lemma}

Let $C([0,1];M)$ denote the collection of continuous curves from $[0,1]$ to $M$. Define the {\it supremum metrics} associated to $d_F$ and $\hat{d}_F$ by
\[
d^*_F(\gamma_1,\gamma_2):=\sup_{t\in [0,1]}d_F(\gamma_1(t),\gamma_2(t)),\quad \hat{d}^*_F(\gamma_1,\gamma_2):=\sup_{t\in [0,1]}\hat{d}_F(\gamma_1(t),\gamma_2(t)).
\]

\begin{theorem}\label{basictopolgyunfirom}
Let $(M,F)$ be a forward complete Finsler manifold. Thus, $(C([0,1];M),d_F^*)$ is a separable forward complete asymmetric metric space. In particular, the forward topology coincides with the symmetrized topology, which  is exactly the metric topology of $\hat{d}^*_F$ (i.e., the compact-open topology). Moreover, the evaluation map $e_t: \gamma\mapsto\gamma(t)$ is continuous from $(C([0,1];M),d_F^*)$ to $M$.
\end{theorem}
\begin{proof}[Sketch of the proof] 
Given $\gamma\in C([0,1];M)$, the compactness of $\gamma([0,1])$ yields $R>0$ such that $\gamma(t)\in B^+_{\gamma(0)}(R)$.   Thus, for  any $\varepsilon>0$ and  for any $\xi\in \mathcal {B}^+_{\gamma}(\varepsilon):=\{\zeta\in C([0,1];M)\,|\ d_F^*(\gamma,\zeta)<\varepsilon\}$, the triangle inequality of $d_F$ yields $\xi(t)\in B^+_{\gamma(0)}(R+\varepsilon)$ for $t\in [0,1]$. Set $\theta:=\lambda_F(  B^+_{\gamma(0)}(R+\varepsilon) )<\infty$. Thus, an argument similar to that of \eqref{reernoncurve} yields
\[
d^*_F(\xi,\zeta)\leq \theta d^*_F(\zeta, \xi),\quad \forall\,\xi,\zeta\in  \mathcal {B}^+_{\gamma}(\varepsilon).
\]
 By this observation, it is not hard to show that the forward topology coincides with the symmetrized topology, which is the metric topology   induced by $\hat{d}^*_F$. Consequently, the rest of the theorem follows from the standard theory in the symmetric case directly (cf. \cite[$\S$2.2]{AGS2}).
\end{proof}


\subsection{Wasserstein spaces over Finsler manifolds}\label{wassspaceoverfinsler}
  This subsection is devoted to the investigation of Wasserstein spaces in the context of Finsler manifolds.
See \cite{KZ,OS,OZ}, etc., for more details.

Let $(M,F)$ be a forward complete Finsler manifold.
Given $\mu,\nu\in \mathscr{P}(M)$, let $\Pi(\mu, \nu)$ denote the collection of {\it transference plans} from $\mu$ to $\nu$, i.e.,
\[
\Pi(\mu,\nu):=\left\{ \pi\in \mathscr{P}{(M\times M)}\,|\ \mathfrak{p}^1_\sharp \pi=\mu,\  \mathfrak{p}^2_\sharp \pi=\nu \right\},
\]
where {$\mathfrak{p}^i:M_1\times M_2\rightarrow M_i$} is the $i$-th natural projection for $i=1,2$.
Given $p\in [1,\infty)$, the {\it Wasserstein distance of order $p$} from $\mu$ to $\nu$ is defined as
\[
W_p(\mu,\nu):=\inf_{\pi\in \Pi(\mu,\nu)}\left( {\int_{M\times M}} d_F(x,y)^p{\dd}\pi(x,y)\right)^{\frac1p}.
\]
Owing to  \cite[Theorem 4.1]{Vi}, there always exists a transference plan $\tilde{\pi}\in \Pi(\mu,\nu)$ such that
\[
W_p(\mu,\nu)=\left(\int_{M\times M} d_F(x,y)^p {\dd}\tilde{\pi}(x,y)                  \right)^{\frac1p}.
\]
Such a $\tilde{\pi}$ is called an {\it optimal transference plan}  from $\mu$ to $\nu$ (with respect to $d_F$).
Given a fixed point $\star\in M$, set
\[
\mathscr{P}_p(M):=\left\{\mu\in \mathscr{P}(M)\,\Big|\ \int_M  {\hat{d}_F}(\star,x)^p{\dd}\mu(x)<\infty \right\}.
\]
The triangle inequality of $d_F$ implies that $\mathscr{P}_p(M)$ is independent of the choice of $\star$.

\begin{theorem}[\cite{V,Vi,KZ}]\label{wassdistacbasisthe} Given $p\in [1,\infty)$, the $p$-Wasserstein space
$(\mathscr{P}_p(M),W_p)$ is an  asymmetric metric space, i.e., for any $\mu,\nu,\upsilon\in \mathscr{P}_p(M)$,
\begin{itemize}
\item[(i)]  $W_p(\mu,\nu)$ is finite$;$

\item[(ii)] $W_p(\mu,\nu)\geq 0,$ with equality if and only if $\mu=\nu;$

\item[(iii)] $W_p(\mu,\nu)\leq W_p(\mu,\upsilon)+W_p(\upsilon,\nu).$

\end{itemize}
In particular, 
$W_q\leq W_p$ for any $1\leq q\leq p$.
Moreover, $W_1$ is the Kantorovich-Rubeinstein distance, i.e.,
\begin{equation}
 {W_1(\mu,\nu)}=\sup_{\psi\in \Lip_1(M)}\left\{ \int_{M}\psi(y)\,{\dd}\nu(y)-\int_M\psi(x)\,{\dd}\mu(x)  \right\},\quad \forall\,\mu,\nu\in \mathscr{P}_1(M),\label{Kantorovich-Rubeinstein}
\end{equation}
where ${\Lip}_1(M):=\{f\in C(M)\,|\, f(y)-f(x)\leq  {d_F(x,y)}  \}$.
\end{theorem}

Moreover, it is not hard to show that the reverse of Wasserstein
distance $\overleftarrow{W_p} $ is exactly the Wasserstein
distance induced by $\overleftarrow{d_F}$, i.e.,
\[
W_p(\nu,\mu)=\overleftarrow{W_p}(\mu,\nu)=\inf_{\pi\in \Pi(\mu,\nu)}\left( {\int_{M\times M}} \overleftarrow{d_F}(x,y)^p{\dd}\pi(x,y)\right)^{\frac1p}.
\]

In the sequel, $\mathscr{P} ({M})$ is equipped with the narrow topology while
$\mathscr{P}_p( {M})$ is endowed by the symmetrized topology $\hat{\mathcal {T}}$ induced by the symmetrized metric $\hat{W}_p$ (see Theorem \ref{asymmetrtopol}).  In particular,  the convergence in $\hat{\mathcal {T}}$ implies the narrow convergence.

\begin{proposition}\label{wasstopologequaivel}
Let $(M,F)$ be a forward complete Finsler manifold and let $p\in [1,\infty)$. If a sequence $(\mu_k)$ converges to $\mu$ in $\mathscr{P}_p( {M})$, then $(\mu_k)$ narrowly converges to $\mu$.

\end{proposition}
\begin{proof}It suffices to show that $(\mu_k)$ is tight.
In fact, if $(\mu_k)$ is tight, then there is a subsequence $(\mu_{k'})$ such that $(\mu_{k'})$ narrowly converges to some probability measure $\tilde{\mu}$. Thus, the lower semicontinuity of $W_p$ with respect to the narrow convergence (cf.\,\cite[Lemma C.6]{KZ})  yields
$W_p(\tilde{\mu},\mu)\leq \liminf_{k'\rightarrow \infty}W_p(\mu_{k'},\mu)=0$, i.e., $\tilde{\mu}=\mu$. And an easy contradiction argument furnishes that the whole sequence $(\mu_k)$ is narrowly convergent to $\mu$ and hence, the proposition follows.

In the sequel, we show the tightness of $(\mu_k)$. 
Theorem \ref{wassdistacbasisthe} implies  $\hat{W}_1(\mu,\mu_k)\leq \hat{W}_p(\mu,\mu_k)$.
Thus $(\mu_k)$ is a Cauchy sequence for $\hat{W}_1$ and hence, a forward Cauchy sequence for $W_1$.
Therefore, for any $\varepsilon>0$ and $\ell\in \mathbb{N}$, there exists a $N=N(\varepsilon,\ell)>0$ such that $W_1(\mu_N,\mu_k)<2^{-2\ell-2}\varepsilon^2$ for any $k>N$.
 Thus,  for any $k\in \mathbb{N}$, there is $j\in \{1,\ldots,N\}$ satisfying
 \begin{equation}\label{controlWptoinetaulty}
 W_1(\mu_j,\mu_k)<2^{-2\ell-2}\varepsilon^2.
 \end{equation}

 Since the finite set $\{\mu_1,\ldots,\mu_N\}$ is always tight, there is a compact set $K\subset M$ with $\mu_j[M\backslash K]<2^{-\ell-1}\varepsilon$ for all $j\in \{1,\ldots,N\}$. This compact set $K$ can be covered by a finite number of small forward balls
\[
K\subset \bigcup_{1\leq i\leq m(\ell)}B^+_{x_i}(2^{-\ell-1}\varepsilon)=:U_\ell.
\]
Now set $U_\ell^\varepsilon:=\{x\in X\,|\ d_F(U_\ell,x)<2^{-\ell-1}\varepsilon\}\subset \cup_{1\leq i\leq m(\ell)}B^+_{x_i}(2^{-\ell}\varepsilon)$ and
$\phi(x):=\min\left\{0,\  \frac{d_F(U_\ell,x)}{2^{-\ell-1}\varepsilon} -1\right\}$.
Thus, $\textbf{1}_{U_\ell}\leq -\phi\leq \textbf{1}_{U_\ell^\varepsilon}$ and $2^{-\ell-1}\varepsilon\phi$ belongs to $\text{Lip}_1(M)$. Given $k\in \mathbb{N}$, choose $j\in \{1,\ldots,N\}$ such that $\mu_j$ satisfies (\ref{controlWptoinetaulty}).  Then
\eqref{Kantorovich-Rubeinstein} together with (\ref{controlWptoinetaulty}) furnishes
\begin{align*}\label{w1wpcontal}
&\mu_k\left[  \bigcup_{1\leq i\leq m(\ell)}B^+_{x_i}(2^{-\ell}\varepsilon) \right] \geq \mu_k[U^\varepsilon_\ell]\geq\int_M -\phi d\mu_k= \int_M -\phi d\mu_j-\left(\int_M \phi d\mu_k -\int_M \phi d\mu_j  \right)\\
&\geq \mu_j[U_\ell]- \frac{W_1(\mu_j,\mu_k)}{2^{-\ell-1}\varepsilon} \geq \mu_j[K]-\frac{W_1(\mu_j,\mu_k)}{2^{-\ell-1}\varepsilon}\geq 1-2^{-\ell}\varepsilon,
\end{align*}
which indicates
\[
\mu_k\left[ M\backslash \bigcup_{1\leq i\leq m(\ell)}B^+_{x_i}(2^{-\ell}\varepsilon) \right]\leq 2^{-\ell}\varepsilon,\quad \forall \,k\in \mathbb{N}.
\]
Now set
$K_\varepsilon:=\cap_{1\leq \ell\leq \infty} \cup_{1\leq i\leq m(\ell)} \overline{B^+_{x_i}(2^{-\ell}\varepsilon)}$.
Thus, for any $k\in \mathbb{N}$, we have $\mu_k\left[ M\backslash K_\varepsilon \right]\leq \varepsilon$.
It remains to show that $K_\varepsilon$ is compact. In fact, for any small $\delta>0$, choose an $\ell>0$ such that $2^{-\ell}\varepsilon<\delta$ and then
\[
K_\varepsilon\subset \bigcup_{1\leq i\leq m(\ell)} \overline{B^+_{x_i}(2^{-\ell}\varepsilon)}\subset\bigcup_{1\leq i\leq m(\ell)} {B^+_{x_i}(\delta)},
\]
 which implies that $K_\varepsilon$ is forward totally bounded in $(M, {d}_F)$.  Owing to \cite[Theorem 2.9]{KZ}, the closedness of $K_\varepsilon$ implies its compactness. Thus, the tightness of $(\mu_k)$ follows.
\end{proof}

It is remarkable that
the forward topology induced by $W_p$ may not coincide with the backward one.
\begin{example}\label{wpcontfunk}
Let $(\mathbb{B}^n, d_F)$ be the Funk space defined by \eqref{distFunk}. For any $p\in [1,\infty)$, there exist a sequence of probability measures $(\mu_k)\subset \mathscr{P}_p(\mathbb{B}^n)$ and a probability measure $\mu\notin \mathscr{P}_p({\mathbb{B}^n})$  such that
\[
\lim_{k\rightarrow \infty}W_p(\mu,\mu_{k})=0,\quad  W_p(\mu_k,\mu)=\infty,\ \forall\,k\in \mathbb{N}.
\]

In fact, let $\mu_k:={\gamma_k}_{\sharp}\mathscr{L}|_{[0,1]}$ and $\mu:=\gamma_{\sharp}\mathscr{L}|_{[0,1]}$, where
 $\gamma_k$ and $\gamma$ are defined by
 \begin{alignat*}{2}
\gamma_k(t) &:=\left(1-e^{\frac{1}{t-1}+1},0,\ldots,0\right), \ \text{ if }t\in \left[0,1-\frac1k\right);   &\qquad \gamma_k(t)&:=\mathbf{0},\ \text{ if }t\in \left[1-\frac1k,1\right];\\
\gamma(t)&:=\left(1-e^{\frac{1}{t-1}+1},0,\ldots,0\right),\ \text{ if }t\in [0,1);   & \gamma(t)&:=\mathbf{0},\ \text{ if }t=1.
 \end{alignat*}
 Thus, by $d_F(x,\mathbf{0})=\ln(1+\|x\|)$ (see \eqref{distFunk}), we have
\begin{align*}
W_p(\mu,\mu_k)^p\leq  {\int_{M\times M}}d_F(x,y)^p{\dd}\pi_k(x,y)=\int^1_0 d_F(\gamma(t),\gamma_k(t))^p{\dd}t=\int^{1}_{1-1/k}d_F(\gamma(t),\mathbf{0})^p{\dd}t\rightarrow 0,
\end{align*}
where $\pi_k=(\gamma,\gamma_k)_\sharp\mathscr{L}|_{[0,1]}\in \Pi(\mu,\mu_k)$. On the other hand, suppose by contradiction that  there  exists some $N_0>0$ such that $W_p(\mu_{N_0},\mu)<\infty$, which together with  the triangle inequality of $W_p$ implies
\begin{equation}\label{w0gamafini}
W_p(\delta_{\mathbf{0}},\mu)\leq W_p(\delta_{\mathbf{0}},\mu_{N_0})+W_p(\mu_{N_0},\mu)<\infty,
\end{equation}
where $\delta_{\mathbf{0}}$ is the Dirac mass  concentrated on $\mathbf{0}$.
However, due to $d_F(\mathbf{0},x)=-\ln(1-\|x\|)$, we get
\begin{align*}
W_p(\delta_{\mathbf{0}},\mu)^p= {\int_{M\times M}}d_F(x,y)^p\,\delta_{\mathbf{0}}\otimes \mu({\dd}x\,{\dd}y)= {\int_M } d_F(\mathbf{0},y)^p{\dd}\mu(y)=\int_0^1d_F(\mathbf{0},\gamma(t))^p{\dd}t=\infty,
\end{align*}
which contradicts \eqref{w0gamafini}.  Hence,  $W_p(\mu_k,\mu)=\infty$ for all $k\in \mathbb{N}$.
\end{example}

In view of Example \ref{wpcontfunk}, it seems impossible to find an accurate relation between $W_p(\mu,\nu)$ and $W_p(\nu,\mu)$. However, if the reversibility satisfies a concavity property,
we have the following result (cf. \cite[Theorem 2.23,\, Lemma 4.4]{KZ}).
\begin{theorem}[\cite{KZ}]\label{reversbilityofW}Let $(M,F)$ be a forward complete Finsler manifold and let $\star\in M$ be a fixed point. Thus, there  exists a nondecreasing function (dependent on $\star$) $\Theta:[0,\infty)\rightarrow [1,\infty)$ such that
\begin{equation}\label{Thetacontrll}
\lambda_{d_F}\left(\overline{B^+_\star(r)}\right)\leq \Theta(r), \quad \forall\,r>0.
\end{equation}
where
\[
\lambda_{d_F}\left( \overline{B^+_\star(r)} \right)
 :=\inf\left\{ \lambda\geq1 \,\Big|\,
 d_F(x,y)\leq \lambda \, d_F(y,x) \text{ for any } x,y \in \overline{B^+_\star(r)} \right\}.
 \]

 Moreover,
given $1\leq q\leq p<\infty$, if $\Theta^{ {qp}/{(p-q)}}$ is a concave function, then
\begin{equation}\label{wwcontroll}
 W_{q}(\nu,\mu)\leq  \Theta\bigg(  W_p(\delta_\star,\mu)+W_p(\mu,\nu) \bigg)\,W_p(\mu,\nu),\quad \forall\,\mu,\nu\in \mathscr{P}_p(M).
 \end{equation}
 Here we use the convention that $\Theta^\infty:=\sup_{x\in M}\Theta(d_F(\star,x))$ and  $\Theta^\infty$ is said to be {\it concave} if it is finite.
\end{theorem}

For simplicity of presentation,  in the sequel a triple $(M,\star,F)$ is called a {\it $\Theta$-Finsler manifold} if $(M,F)$ is a Finsler manifold, $\star$ is a fixed point in $M$ and $\Theta$ is a nondecreasing function satisfying \eqref{Thetacontrll}.

\subsection{Structures of absolutely continuous curves in Wasserstein spaces}\label{sacws}

In view of Example \ref{wpcontfunk}, the structure of $(\mathscr{P}_p(M),W_p)$ over a Finsler manifold is much different from the one in the symmetric case.  To begin with,
   we present  the existence of absolutely continuous curves in $(\mathscr{P}_p(M),W_p)$.
 \begin{proposition}
 Let $(M,F)$ be a forward complete Finsler manifold and let $p\in [1,\infty)$. Thus, for every $\mu_0,\mu_1\in \mathscr{P}_p( {M})$, there exists a curve $\mu_t\in \FAC^p([0,1];\mathscr{P}_p( {M}))$ from $\mu_0$ to $\mu_1$. Moreover, if  the supports of $\mu_0,\mu_1$ are both compact, then $\mu_t\in \AC^p([0,1];\mathscr{P}_p( {M}))$.
 \end{proposition}
\begin{proof} According to  \cite[Theorem 4.16]{KZ}, there exists  $\eta\in \mathscr{P}(C([0,1];M))$ such that  $\mu_t=(e_t)_\sharp\eta$ and for any $[s,t]\subset [0,1]$,
\begin{equation}\label{realcontrw}
(t-s)W_p(\mu_0,\mu_1)=W_p(\mu_s,\mu_t)=\int_{C([0,1];M)}d_F(e_s(\gamma),e_t(\gamma)){\dd} \eta(\gamma)<\infty,
\end{equation}
which implies $\mu_t\in \FAC^p([0,1];\mathscr{P}_p( {M}))$. Moreover, if $\supp\mu_0\cup\supp\mu_1$ is compact,
\cite[Lemma D.5]{KZ} yields a compact set $K\subset M$ such that $\gamma([0,1])\subset K$ for $\eta$-a.e. $\gamma\in \supp \eta$. By letting $\theta:=\lambda_F(K)<\infty$ and using \eqref{realcontrw}, we obtain
\begin{align*}
W_p(\mu_t,\mu_s)\leq \int_{\supp \eta}d_F(e_t(\gamma),e_s(\gamma)){\dd} \eta(\gamma)\leq \theta \int_{\supp \eta}d_F(e_s(\gamma),e_t(\gamma)){\dd} \eta(\gamma)=\theta \,W_p(\mu_s,\mu_t),
\end{align*}
which indicates $\mu_t\in \AC^p([0,1];\mathscr{P}_p( {M}))$.
\end{proof}

Now we are going to investigate the metric derivative of absolutely continuous curves in $(\mathscr{P}_p(M),W_p)$.
Since $W_q\leq W_p$ for any $1\leq q\leq p$, the following result is a direct consequence  of   Proposition \ref{baspeedforback}.

\begin{theorem}\label{vilocitythem} Let $(M,F)$ be a forward complete Finsler manifold  and $p\in [1,\infty)$.
 For any $\mu_t\in \FAC^p([0,1];\mathscr{P}_p( {M}))$, i.e.,
 \begin{equation}\label{facinwpspace}
 W_p  {\big(\mu_{t_1},\mu_{t_2} \big)}\leq \int^{t_2}_{t_1} f(r){\dd}r, \quad \text{ for any } [t_1,t_2]\subset [0,1],
 \end{equation}
  for some $f\in L^p([0,1])$,  the forward metric derivative
\[
 |\mu_+'|_q(t) :=\lim_{s \to t} \frac{W_q(\mu_{\min\{s,t\}},\mu_{\max\{s,t\}})}{|t-s|}
\]
exists for $\mathscr{L}^1$-a.e.\ $t \in (0,1)$ and for any $q\in [1,p]$. In particular, $|\mu_+'|_q\in L^p([0,1])$ and satisfies
\[
W_q\big(\mu_{t_1},\mu_{t_2} \big) \leq
 \int^{t_2}_{t_1} |\mu_+'|_q(r) \,{\dd}r, \quad \text{ for any } [t_1,t_2]\subset [0,1].
\]

Provided $\mu_t\in \BAC^p([0,1];\mathscr{P}_p( {M}))$, a similar property holds for the backward metric derivative $|\mu_-'|_q$.
\end{theorem}

Clearly, a $p$-absolutely continuous curve is always continuous in $\mathscr{P}_p( {M})$. But
in view of Example \ref{wpcontfunk}, it seems impossible to check the continuity of a partially absolutely continuous curve in a generic Wasserstein space.
The following result is  the next best thing.
\begin{proposition}\label{compactfacpp}Let $(M,\star,F)$ be a forward complete $\Theta$-Finsler manifold and let $p\in [1,\infty)$.
Suppose that $\Theta^{ {qp}/{(p-q)}}$ is a concave function for some $q\in [1,p]$.
 Thus,
 \[
 \FAC^p([0,1];\mathscr{P}_p(M))\cup \BAC^p([0,1];\mathscr{P}_p(M))\subset \AC^p([0,1];\mathscr{P}_q(M)).
 \] Hence,
 for every  $\mu_t\in \FAC^p([0,1];\mathscr{P}_p(M))\cup\BAC^p([0,1];\mathscr{P}_p(M))$,
   it is a continuous curve in $(\mathscr{P}_q(M),W_q)$ and  $\{\mu_t\,|\ t\in [0,1]\}$ is a compact set in $\mathscr{P}(M)$.
\end{proposition}
\begin{proof}
Given $\mu_t\in \FAC^p([0,1];\mathscr{P}_p(M))$, there is $f\in L^p([0,1])$ satisfying \eqref{facinwpspace}. Thus,
 \eqref{wwcontroll} together with   the triangle inequality of $W_p$ and \eqref{facinwpspace} implies
\begin{align*}
W_q(\mu_{t_2},\mu_{t_1})&\leq \Theta\left(  W_p(\delta_\star,\mu_0)+2\int^1_0 f(r){\dd}r \right)\,W_p(\mu_{t_1},\mu_{t_2})
\leq \Theta\left(  W_p(\delta_\star,\mu_0)+2\|f\|^p_{L^p} \right)\int^{t_2}_{t_1} f(r){\dd}r
\end{align*}
for any $[t_1,t_2]\subset [0,1]$, which combined with  Theorem \ref{vilocitythem} yields $\mu_t\in \AC^p([0,1];\mathscr{P}_q( {M}))$. Thus, it is continuous in the symmetrized space $(\mathscr{P}_q( {M}),\hat{W}_q)$. Then Proposition \ref{wasstopologequaivel} implies the compactness of $\{\mu_t\,|\ t\in [0,1]\}$  in $\mathscr{P}(M)$.
The case  when  $\mu_t\in \BAC^p([0,1];\mathscr{P}_p(M))$   can be proved in the same way.
\end{proof}


Inspired by \cite[Theorem 5]{Li} and \cite[Theorem 3.1]{Li2}, we obtain the following result, which characterizes the natural of forward absolutely continuous curves by dynamical transference plans.

\begin{theorem}\label{maintheorem}Let $(M,\star,F)$ be a forward complete $\Theta$-Finsler manifold and let $p\in (1,\infty)$.
  Suppose that $\Theta^{qp/(p-q)}$ is a convex function for some $q\in (1,p]$. Thus,
for any $\mu_t\in  \FAC^p([0,1];\mathscr{P}_p(M))$, there exists $\eta\in \mathscr{P}(C([0,1];M))$ such that
\begin{enumerate}[{\rm (i)}]

\item \label{supent1} $\eta$ is concentrated on $\AC^p([0,1];M)$;

\smallskip

\item\label{supent2} $\mu_t=(e_t)_\sharp \eta$ for any $t\in [0,1]$;

\item\label{supent3}  $\ds |\mu'_+|^p_p(t)= \int_{C([0,1];M)} F^p(\gamma'(t)) \,{\dd}\eta(\gamma)$ for $\mathscr{L}^1$-a.e. $t\in (0,1)$.
\end{enumerate}
\end{theorem}

\begin{proof}
For any integer $N\geq 1$, we divide the interval $[0,1]$ into $2^N$ equal parts, and set the nodal parts
\begin{equation}\label{defti}
t^i:=\frac{i}{2^N},\quad i=0,1,\ldots, 2^N.
\end{equation}
Let $M_i$, $i= 0,1,\ldots, 2^N$ be $2^N+1$ copies of $M$ and set
\[
\mathbf{M}:=M_0\times \cdots\times M_{2^N}.
\]

Choose
  optimal transference plans $\pi_{i,i+1}\in \Pi(\mu_{t^i},\mu_{t^{i+1}})$ in $(\mathscr{P}_p(M),W_p)$, $i=0,\ldots,2^{N}-1$. Thus, owing to \cite[Lemma 5.3.4]{AGS}, there exists $\pi_N\in \mathscr{P}(\mathbf{M})$  such that
\begin{equation}\label{distingragofpi_N}
\mathfrak{p}^i_\sharp \pi_N=\mu_{t^i},\quad  \mathfrak{p}^{i,i+1}_\sharp \pi_N=\pi_{i,i+1},
\end{equation}
where $\mathfrak{p}^i:\mathbf{M}\rightarrow M_i$ and $\mathfrak{p}^{i,i+1}:\mathbf{M}\rightarrow M_i\times M_{i+1}$  are the natural projections.

Now define a map $\sigma:\mathbf{M}\rightarrow \mathcal {M}([0,1];M)$ by
\begin{equation}\label{defsigma}
\mathbf{x}=(x_0,\ldots,x_{2^N})\longmapsto\sigma_t(\mathbf{x}):=x_i, \text{ for }t\in [t^i,t^{i+1}),
\end{equation}
and set
\begin{equation}\label{etandef}
\widetilde{\eta}_N:=\sigma_\sharp\pi_N\in \mathscr{P}(\mathcal {M}([0,1];M)).
\end{equation}
The rest of the proof is divided into {six} steps.

\medskip

\noindent \textbf{Step 1.} In this step, we prove that for any $\alpha\in [1,p]$,
\begin{align}
\sup_{1/2^N\leq h<1}\int^{1-h}_0 \left(\frac{{d}_F(\sigma_t(\mathbf{x}),\sigma_{t+h}(\mathbf{x})) }{h }\right)^\alpha{\dd}t&\leq 2^{\alpha+N(\alpha-1)}\sum_{i=0}^{2^N-1}{d}_F(x_i,x_{i+1})^\alpha,\quad \forall\,\mathbf{x}\in \mathbf{M};\label{lapclsigmestimate}\\
\sup_{0<h<1}\int^{1-h}_0 \frac{{d}_F(\sigma_t(\mathbf{x}),\sigma_{t+h}(\mathbf{x}))^\alpha}{h}{\dd}t &\leq \left( 2^\alpha+\frac{2^{\alpha N}}{2^N-1}  \right)\sum_{i=0}^{2^N-1} {d}_F(x_i,x_{i+1})^\alpha,\quad \forall\,\mathbf{x}\in \mathbf{M};\label{forprobb}\\
\int_{\mathbf{M}}d_F(x_i,x_{i+1})^\alpha {\dd}\pi_N(\mathbf{x})&\leq W_p(\mu_{t^i},\mu_{t^{i+1}})^\alpha.\label{forprobb22}
\end{align}

For \eqref{lapclsigmestimate},  if $1/2^N\leq h<1$, then choose the integer $k\geq 1$ such that
\begin{equation}\label{khcontroll}
\frac{k}{2^N}\leq h<\frac{k+1}{2^N}\quad \Longrightarrow \quad \frac{1}{2^N}\leq \min\left\{\frac{h}k,\ 2^{N(\alpha-1)}\frac{h^\alpha}{k^\alpha}\right\}.
\end{equation}
Given $\alpha\in [1,p]$, the triangle inequality of $d_F$ yields
\begin{align}
{d}_F(\sigma_t(\mathbf{x}),\sigma_{t+h}(\mathbf{x}))^{\alpha}\leq  \left[\sum_{i=0}^k  {d}_F(\sigma_{t+t^i}(\mathbf{x}),\sigma_{t+t^{i+1}}(\mathbf{x}))\right]^{\alpha}\leq (k+1)^{{\alpha}-1}\sum_{i=0}^k {d}_F(\sigma_{t+t^i}(\mathbf{x}),\sigma_{t+t^{i+1}}(\mathbf{x}))^{\alpha}.
\label{sumqqesti}
\end{align}
By $1-t^k=(2^N-k)/2^N$ and \eqref{khcontroll} we have
\begin{align}
&\int^{1-h}_0 {d}_F(\sigma_t(\mathbf{x}),\sigma_{t+h}(\mathbf{x}))^{\alpha}{\dd}t\leq \int_0^{1-t^k}{d}_F(\sigma_t(\mathbf{x}),\sigma_{t+h}(\mathbf{x}))^{\alpha}{\dd}t\notag\\
\leq & (k+1)^{\alpha-1}\sum_{i=0}^k \int^{1-t^k}_0{d}_F(\sigma_{t+t^i}(\mathbf{x}),\sigma_{t+t^{i+1}}(\mathbf{x}))^{\alpha}{\dd}t\notag\\
=& (k+1)^{\alpha-1}\sum_{i=0}^k \sum_{j=0}^{2^N-k-1} \int^{t^{j+1}}_{t^j}{d}_F(\sigma_{t+t^i}(\mathbf{x}),\sigma_{t+t^{i+1}}(\mathbf{x}))^{\alpha}{\dd}t\notag\\
=&(k+1)^{\alpha-1}\sum_{i=0}^k \frac1{2^N} \sum_{j=0}^{2^N-k-1}{d}_F( {x_{i+j}}, x_{i+j+1} )^{\alpha}\leq \frac{(k+1)^{\alpha}}{2^N}\sum_{j=0}^{2^N-1}{d}_F(x_j,x_{j+1})^{\alpha}\label{khNinequality}\\
\leq &2^{N(\alpha-1)}h^\alpha\frac{(k+1)^{\alpha}}{k^\alpha}\sum_{j=0}^{2^N-1}{d}_F(x_j,x_{j+1})^{\alpha}\leq 2^{N(\alpha-1)+\alpha} h^\alpha\sum_{j=0}^{2^N-1}{d}_F(x_j,x_{j+1})^{\alpha},\notag
\end{align}
which is exactly \eqref{lapclsigmestimate}. 
Now we show \eqref{forprobb}. Firstly, an argument similar to that of  \eqref{khNinequality} combined with \eqref{khcontroll}  furnish
\begin{equation*}
\int^{1-h}_0  {d}_F(\sigma_t(\mathbf{x}),\sigma_{t+h}(\mathbf{x}))^\alpha{\dd}t\leq h\frac{(k+1)^{\alpha}}{k}\sum_{j=0}^{2^N-1} {d}_F(x_j,x_{j+1})^\alpha,
\end{equation*}
which together with
$\frac{(k+1)^\alpha}{k}\leq 2^\alpha+\frac{2^{N\alpha}}{2^N-1}$ yields
\begin{align}\label{bacisinequl}
\sup_{1/2^N\leq h<1}\int^{1-h}_0 \frac{ {d}_F(\sigma_t(\mathbf{x}),\sigma_{t+h}(\mathbf{x}))^\alpha}{h}{\dd}t \leq \left( 2^\alpha+\frac{2^{\alpha N}}{2^N-1}  \right)\sum_{i=0}^{2^N-1} {d}_F(x_i,x_{i+1})^\alpha.
\end{align}
On the other hand, for $h\in (0,1/2^N)$, since $\sigma_{t+h}(\mathbf{x})=\sigma_t(\mathbf{x})$ if $t\in [t^i,t^{i+1}-h)$, there holds
\begin{align*}
\int^{1-h}_0 {d}_F(\sigma_t(\mathbf{x}),\sigma_{t+h}(\mathbf{x}))^\alpha{\dd}t=\sum_{i=0}^{2^N-2}\int^{t^{i+1}}_{t^i} {d}_F(\sigma_t(\mathbf{x}),\sigma_{t+h}(\mathbf{x}))^\alpha{\dd}t=h\sum_{i=0}^{2^N-2} {d}_F(x_i,x_{i+1})^\alpha,
\end{align*}
which combined with \eqref{bacisinequl} yields \eqref{forprobb}.

As for \eqref{forprobb22}, recall that $\mathfrak{p}^{i,i+1}_\sharp \pi_N=\pi_{i,i+1}\in \Pi(\mu_{t^i},\mu_{t^{i+1}})$ is an optimal transference plan in $(\mathscr{P}_p(M),W_p)$. Then  the H\"older inequality together with $\mathfrak{p}^{i,i+1}(\mathbf{x})=(x_i,x_{i+1})$ implies
\begin{align*}
\int_{\mathbf{M}} {d}_F(x_i,x_{i+1})^\alpha{\dd}\pi_N(\mathbf{x})
\leq  \left(  \int_{M_i\times M_{i+1}} {d}_F(x_i,x_{i+1})^p{\dd}(\mathfrak{p}^{i,i+1}_\sharp \pi_N)(x_i,x_{i+1}) \right)^{\alpha/p}  = {W_p(\mu_{t_i},\mu_{t^{i+1}})^\alpha}.
\end{align*}

\medskip

\noindent \textbf{Step 2.} In this step, we show the tightness of $(\widetilde{\eta}_N)_{N}$ in $\mathcal {M}([0,1];M)$. According to  Remark \ref{anoterhtightorp}, it suffices to construct a function $\Phi:\mathcal {M}([0,1];M)\rightarrow [0,\infty]$ satisfying the following two properties:
\begin{enumerate}[{\rm {\quad} (a)}]

\item\label{basicapoer1} for any $c\geq 0$, the sublevel $\lambda_c(\Phi):=\{\gamma\in \mathcal {M}([0,1];M)\,|\ \Phi(\gamma)\leq c \}$ is compact in $\mathcal {M}([0,1];M)$;

\smallskip

\item\label{basicapoer2} $\ds \sup_{N\in \mathbb{N}}\int_{\mathcal {M}([0,1]; {M})}\Phi(\gamma){\dd}\widetilde{\eta}_N(\gamma)<\infty$.

\end{enumerate}

\medskip

\textbf{Construction of $\Phi$:} now we construct the function $\Phi$.
The assumption together with Proposition \ref{compactfacpp} implies that $\mu_t\in \FAC^p([0,1];\mathscr{P}_p(X))$  is $p$-absolutely continuous in    $(\mathscr{P}_q(X), \hat{W}_q)$ and
$\mathscr{A}:=\{\mu_t\,|\ t\in [0,1]\}$ is a compact set   in $\mathscr{P}(X)$. Thus,  there exists $f\in L^p([0,1])$ such that for any $[s,t]\subset[0,1]$,
\begin{equation}\label{facinwpspace2}
  W_p(\mu_s,\mu_t)\leq  \int^t_s  f(r){\dd}r,\quad   W_q(\mu_s,\mu_t)+ W_q(\mu_t,\mu_s)\leq \int^t_s  f(r){\dd}r.
 \end{equation}
Furthermore, the compactness of
  $\mathscr{A}$   in $\mathscr{P}(M)$ implies its  tightness by Theorem \ref{grenarlizedpROKTHE} (by considering $(M,\hat{d}_F)$). Then Remark \ref{anoterhtightorp} together with Proposition \ref{completedifin} furnishes a function $\psi:M\rightarrow [0,\infty]$ satisfying
\begin{itemize}

\item for any $c\geq 0$, the sublevel $\lambda_c(\psi):=\{x\in M\,|\ \psi(x)\leq c \}$ is compact in $(M,d_F)$;

\smallskip

\item $\ds C_1:=\sup_{t\in [0,1]}\int_{X}\psi(x){\dd}\mu_t(x)<\infty$.

\end{itemize}
In particular, it follows from \cite[Remark 5.1.5,\ (5.1.13)]{AGS}  that   $\psi$ is lower semicontinuous.
Now we define $\Phi:\mathcal {M}([0,1];M)\rightarrow [0,\infty]$ as
\begin{equation}\label{defofPHI}
\Phi(\gamma):=\int^1_0 \psi(\gamma(t)){\dd}t+\sup_{0<h<1}\int^{1-h}_0 \frac{d_F(\gamma(t),\gamma(t+h))}{h}{\dd}t.
\end{equation}
Thus Fatou's lemma implies that $\Phi$ is a lower semicontinuous function.

\smallskip

 \textbf{Property \eqref{basicapoer1}:} we show that $\Phi$ satisfies \eqref{basicapoer1} (see the beginning of Step 2).
In fact, there holds
\[
\lim_{h\rightarrow0^+}\sup_{\gamma\in  \lambda_c(\Phi)}\int^{1-h}_0 {d}_F(\gamma(t),\gamma(t+h)) {\dd}t\leq \lim_{h\rightarrow0^+} ch=0.
\]
Besides, note that $\lambda_c(\psi)$ is compact and
\[
\int^1_0 \psi(\gamma(t)){\dd}t\leq \Phi(\gamma)\leq c,\ \forall\,\gamma \in \lambda_c(\Phi)\quad \Longrightarrow \quad \sup_{\gamma\in \lambda_c(\Phi) }\int^1_0 \psi(\gamma(t)){\dd}t\leq c<\infty.
\]
Thus, the precompactness of $\lambda_c(\Phi)$ in $\mathcal {M}([0,1];M)$ follows by  Theorem \ref{tightcompreversible} directly.

Given a sequence $(\gamma_n)\subset \lambda_c(\Phi)$, the precompactness of $\lambda_c(\Phi)$ yields a limit point $\gamma\in \mathcal {M}([0,1];M)$, while the lower semicontinuity of $\Phi$ implies
\[
\Phi(\gamma)\leq \liminf_{n\rightarrow \infty}\Phi(\gamma_n)\leq c\quad \Longrightarrow \quad \gamma\in \lambda_c(\Phi).
\]
Therefore, $\lambda_c(\Phi)$ is compact in $\mathcal {M}([0,1];M)$, i.e., Property \eqref{basicapoer1} follows.

\smallskip

 \textbf{Property \eqref{basicapoer2}:} in order to show \eqref{basicapoer2}, we firstly claim
\begin{equation}\label{boundsupNdoublinter}
\sup_{N\in \mathbb{N}}\int_{\mathcal {M}([0,1];M)}\int^1_0  \psi(\gamma(t)) {\dd}t{\dd}\widetilde{\eta}_N(\gamma)<\infty.
\end{equation}
In fact,  for $t\in [t^i,t^{i+1})$, by (\ref{defsigma}) and (\ref{distingragofpi_N})  we have
\begin{align*}
\int_{\mathbf{M}} \psi(\sigma_t(\mathbf{x})) {\dd}\pi_N(\mathbf{x})
=\int_{\mathbf{M}} \psi(\mathfrak{p}^i(\mathbf{x}))  {\dd}\pi_N(\mathbf{x})
=\int_{M}  \psi(x) {\dd}\mu_{t^i}(x),
\end{align*}
which combined with \eqref{etandef}  furnishes
\begin{align}
&\int^1_0 \int_{\mathcal {M}([0,1];M)}  \psi(\gamma(t)) {\dd}\widetilde{\eta}_N(\gamma){\dd}t= \int^1_0 \int_{\mathcal {M}([0,1];M)}  \psi(e_t(\gamma)) {\dd}\widetilde{\eta}_N(\gamma){\dd}t =\int^1_0 \int_{\mathbf{M}} \psi(\sigma_t(\mathbf{x})) {\dd}\pi_N(\mathbf{x}){\dd}t\notag\\
&=\sum_{i=0}^{2^N-1}\int^{t^{i+1}}_{t^i}\int_{\mathbf{M}} \psi(\sigma_t(\mathbf{x}))  {\dd}\pi_N(\mathbf{x}){\dd}t=\sum_{i=0}^{2^N-1}\int^{t^{i+1}}_{t^i}\int_{M} \psi(x)  {\dd}\mu_{t^i}(x){\dd}t\label{etaNtNconverse}\\
&=\frac1{2^N}\sum_{i=0}^{2^N-1}\int_{M}  \psi(x)  {\dd}\mu_{t^i}(x)\leq \frac1{2^N}\sum_{i=0}^{2^N-1} C_1 = C_1.\notag
\end{align}
Thus, (\ref{boundsupNdoublinter}) follows by Fubini's theorem.
Secondly,
by \eqref{forprobb}, \eqref{forprobb22} and \eqref{facinwpspace2}, we have
\begin{align}
&\int_{\mathcal {M}([0,1];M)}\sup_{0<h<1}\int^{1-h}_0 \frac{ {d}_F(\gamma(t),\gamma(t+h)) }{h}{\dd}t {\dd}\widetilde{\eta}_N(\gamma)= \int_{\mathcal {M}([0,1];M)}\sup_{0<h<1}\int^{1-h}_0  \frac{ {d}_F(e_t(\gamma) ,e_{t+h}(\gamma)) }{h}{\dd}t {\dd}\widetilde{\eta}_N(\gamma)\notag\\
&= \int_{\mathbf{M}}\sup_{0<h<1}\int^{1-h}_0 \frac{ {d}_F(\sigma_t(\mathbf{x}),\sigma_{t+h}(\mathbf{x})) }{h}{\dd}t {\dd}\pi_N(\mathbf{x}) \leq  \left( 2 +\frac{2^{  N }}{2^N-1}  \right)\sum_{i=0}^{2^N-1}\int_{\mathbf{M}} {d}_F(x_i,x_{i+1})  {\dd}\pi_N(\mathbf{x})\label{keylatboelone}\\
&\leq \left( 2 +\frac{2^{  N }}{2^N-1}  \right)   \sum_{i=0}^{2^N-1}W_p(\mu_{t_i},\mu_{t^{i+1}}) \leq \left( 2 +\frac{2^{  N }}{2^N-1}  \right) \sum_{i=0}^{2^N-1} \int^{t^{i+1}}_{t^i}f(t){\dd}t
\leq 4\int^1_0 f (t){\dd}t<\infty,\notag
\end{align}
which together with \eqref{defofPHI} and \eqref{boundsupNdoublinter} furnishes
 Property \eqref{basicapoer2}.

\smallskip

Since Properties \eqref{basicapoer1},\eqref{basicapoer2} are true,   we obtain the tightness of $(\widetilde{\eta}_N)_N$ in $\mathcal {M}([0,1];M)$.


\medskip

\noindent \textbf{Step 3.}
As $(\widetilde{\eta}_N)$ is tight, let $\widetilde{\eta}$ be an arbitrary narrow limit of $(\widetilde{\eta}_N)$, i.e., a subsequence  $\widetilde{\eta}_{N_k}\Rightarrow \widetilde{\eta}$  in $\mathscr{P}(\mathcal {M}([0,1];M))$ (see \eqref{narrwconver}).
In this step, we show that $\widetilde{\eta}$ is concentrated on BV right-continuous curves.

Given a curve $\gamma:[0,1]\rightarrow M$, let the pointwise variation and essential variation be denoted by
\begin{align*}
\Vv(\gamma;[0,1])&:=\sup\left\{   \sum_{i=0}^{n-1} d_F(\gamma(s_i),\gamma(s_{i+1}))\,\Big|\,0=s_0<s_2<\cdots<s_{n}=1    \right\};\\
\eV(\gamma;[0,1])&:=\inf\Big\{ \Vv(\zeta;[0,1])\,\big|\, \zeta(s)=\gamma(s) \text{ for $\mathscr{L}^1$-a.e. $s\in [0,1]$}   \Big\}.
\end{align*}
Also define a function $L_N:\mathcal {M}([0,1];M)\rightarrow [0,\infty)$ by
\[
L_N(\gamma)=\left \{
\begin{array}{lll}
\ds \eV(\gamma;[0,1]),&&\text{ if }\gamma\in \supp \eta_N;\\
\\
0,&&\text{ otherwise. }
\end{array}
\right.
\]

If $\gamma(t)=\sigma_t({\mathbf{x}})$  for $\mathscr{L}^1$-a.e. $t\in [0,1]$, by \eqref{defsigma} we obtain
\[
\eV(\gamma;[0,1])=\Vv(\sigma(\mathbf{x});[0,1])=\sum_{j=0}^nd_F(x_j,x_{j+1}),
\]
which combined with \eqref{etandef} and \eqref{forprobb22} implies
\begin{align}
&\sup_{N\in \mathbb{N}}\int_{\mathcal {M}([0,1];M)}L_N(\gamma){\dd}\widetilde{\eta}_N(\gamma)=\sup_{N\in \mathbb{N}}\int_{\mathcal {M}([0,1];M)}\sum_{j=0}^{2^N-1}d_F(x_j,x_{j+1}){\dd}\widetilde{\eta}_N(\gamma)\notag\\
=&\sup_{N\in \mathbb{N}}\int_{\mathbf{M}}\sum_{j=0}^{2^N-1}d_F(x_j,x_{j+1}){\dd}\pi_N(\mathbf{x})\leq \sup_{N\in \mathbb{N}}\left(\sum_{j=0}^{2^N-1} W_p(\mu_{t^j},\mu_{t^{j+1}})\right)=\int^1_0f(t){\dd}t<\infty.\label{controllLnboundedness}
\end{align}

By passing a subsequence  and  Lemma \ref{mcontress}, we may assume that   for $\widetilde{\eta}$-a.e. $\gamma\in \supp\widetilde{\eta}$, there is $\gamma_{N_k}\in \supp\widetilde{\eta}_{N_k}$ with
\[
\lim_{k\rightarrow \infty}\mathfrak{d}_1(\gamma,\gamma_{N_k})=0,\qquad \sup_{k\in \mathbb{N}}L_{N_k}(\gamma_{N_k})=: C_2<\infty.
 \]
 Fix $\gamma\in \supp\widetilde{\eta}$ and $\gamma_{N_k}\in \supp\widetilde{\eta}_{N_k}$.
 By \eqref{conveginmeasure} and extracting a further subsequence, we may also assume that $\gamma_{N_k}(t)$ pointwisely converges to $\gamma(t)$ for $\mathscr{L}^1$-a.e. $t\in [0,1]$.

Owing to the discreteness of $\gamma_{N_k}\in \supp\widetilde{\eta}_{N_k}$, we can choose the piecewise constant right-continuous representative of $\gamma_{N_k}$, which is still denoted by $\gamma_{N_k}$. Thus,
\begin{equation}\label{vargammank}
\eV(\gamma_{N_k})=\Vv(\gamma_{N_k})\leq C_2.
\end{equation}

For each $k$, define an incresaing function $v_k:[0,1]\rightarrow [0,C_2]$ by $v_k(t):=\Vv(\gamma_{N_k};[0,t])$. By extracting a subsequence, it follows from \eqref{vargammank} and Helly's selection theorem that $(v_k )$ pointwisely converges to an increasing $v : [0,1]\rightarrow [0,C_2]$. Since the set of discontinuity points of $v$ is at most countable, we can redefine a right-continuous function $\bar{v}$ by $\bar{v}(t):=\lim_{s\rightarrow t^+}v(s)$.
Thus, by observing
\begin{equation*}
d_F(\gamma_{N_k}(s),\gamma_{N_k}(t))\leq v_k(t)-v_k(s),\quad \forall\,[s,t]\subset [0,1],
\end{equation*}
we derive
\begin{equation}\label{basiccontrllv}
d_F(\gamma(s),\gamma(t))\leq \bar{v}(t)-\bar{v}(s), \text{ for $\mathscr{L}^1$-a.e. $s,t\in [0,1]$ with $s\leq t$}.
\end{equation}
Owing to Proposition \ref{completedifin},   we can choose a  representative of $\gamma$  defined by $\bar{\gamma}(t):=\lim_{s\rightarrow t^+}\gamma(s)$. Clearly, $\bar{\gamma}$ is right-continuous and moreover,  \eqref{basiccontrllv} yields
\begin{equation}\label{boundedvariationcrv}
d_F(\bar{\gamma}(s),\bar{\gamma}(t))\leq \bar{v}(t)-\bar{v}(s),\quad  \forall\,[s,t]\subset [0,1].
\end{equation}

From above, we see that for $\widetilde{\eta}$-a.e. $\gamma \in \supp\widetilde{\eta}$, there is a right-continuous representative $\bar{\gamma}$ with \eqref{boundedvariationcrv} which  is factually continuous except at most a countable set.

\medskip

\noindent \textbf{Step 4.} In this step, we show Statement \eqref{supent1}.
Now define a sequence of lower semicontinuous functions $f_N:\mathcal {M}([0,1];M)\rightarrow [0,\infty]$ by
\[
f_N(\gamma):=\sup_{1/2^N\leq h<1}\int^{1-h}_0 \left(\frac{ {d}_F(\gamma(t),\gamma(t+h)) }{h }\right)^p{\dd}t.
\]
Since $[1/2^N,1)\subset [1/2^{N+1},1)$, they satisfy the monotonicity property
\begin{equation}\label{fnfn+1estimate}
f_N(\gamma)\leq f_{N+1}(\gamma),\quad \forall\,\gamma\in \mathcal {M}([0,1];M).
\end{equation}

Moreover, we claim that there exists a constant $C_3>0$ such that
\begin{equation}\label{lpestimatefn}
 \int_{\mathcal {M}([0,1];M)} f_N(\gamma){\dd}\widetilde{\eta}_N(\gamma)\leq C_3,\quad \forall\,N\in \mathbb{N}.
\end{equation}
In fact, a modification of \eqref{keylatboelone} together with \eqref{lapclsigmestimate} and \eqref{forprobb22}   yields
\begin{align*}
& \int_{\mathcal {M}([0,1];M)} f_N(\gamma){\dd}\widetilde{\eta}_N(\gamma)= \int_{\mathcal {M}([0,1];M)}\sup_{1/2^N\leq h<1}\int^{1-h}_0  \left(\frac{ {d}_F(e_t(\gamma),e_{t+h}(\gamma)) }{h }\right)^p{\dd}t {\dd}\widetilde{\eta}_N(\gamma)\notag\\
=& \int_{\mathbf{M}}\sup_{1/2^N\leq h<1}\int^{1-h}_0 \left(\frac{ {d}_F(\sigma_t(\mathbf{x}),\sigma_{t+h}(\mathbf{x}))}{h}\right)^p{\dd}t {\dd}\pi_N(\mathbf{x}) \leq  2^{p+N(p-1)}\sum_{i=0}^{2^N-1}\int_{\mathbf{M}} {d}_F(x_i,x_{i+1})^p {\dd}\pi_N(\mathbf{x})\notag\\
\leq &2^{p+N(p-1)}   \sum_{i=0}^{2^N-1}W_p(\mu_{t_i},\mu_{t^{i+1}})^{p}\leq 2^p \int^1_0 f^p(t){\dd}t<\infty.\label{derivatproof}
\end{align*}
Hence, \eqref{lpestimatefn} follows by choosing $C_3:=[2  \|f\|_{L^p([0,1])}]^p$.
Thus, if $N_k\geq N$, \eqref{lpestimatefn} and \eqref{fnfn+1estimate} furnish
\[
\int_{\mathcal {M}([0,1];M)}f_N(\gamma){\dd}\widetilde{\eta}_{N_k}(\gamma)\leq \int_{\mathcal {M}([0,1];M)}f_{N_k}(\gamma){\dd}\widetilde{\eta}_{N_k}(\gamma)\leq  C_3,
\]
 which together with the lower semicontinuity of $f_N$ and \cite[Lemma 5.1.7]{AGS} yields
\[
\int_{\mathcal {M}([0,1];M)}f_N(\gamma){\dd}\widetilde{\eta}(\gamma)\leq C_3,\quad \forall\,N\in \mathbb{N}.
\]
Thus, the monotone convergence theorem indicates
$\int_{\mathcal {M}([0,1];M)}\sup_{N\in \mathbb{N}}f_N(\gamma){\dd}\widetilde{\eta}(\gamma)\leq C_3$ and hence,
\begin{equation*}\label{fnfinite}
\sup_{0<h<1}\int^{1-h}_0 \left(\frac{ {d}_F(\gamma(t),\gamma(t+h)) }{h }\right)^p{\dd}t=\sup_{N\in \mathbb{N}}f_N(\gamma)<\infty, \quad \text{for $\widetilde{\eta}$-a.e. $\gamma\in \mathcal {M}([0,1];M)$},
\end{equation*}
which implies
\[
\limsup_{h\rightarrow 0^+}\left\| \frac{d_F(\gamma(\cdot),\gamma(\cdot+h))}{h}    \right\|_{L^p([0,1])}<\infty, \quad \text{for $\widetilde{\eta}$-a.e. $\gamma\in \mathcal {M}([0,1];M)$}.
\]
Therefore, in view of the end of Step 3 and Theorem \ref{strongcontrall}, for   $\widetilde{\eta}$-a.e. $\gamma \in \supp\widetilde{\eta}$, there is a representative $\bar{\gamma}\in \AC^p([0,1];M)$.
Let $T:C([0,1];M)\rightarrow \mathcal {M}([0,1];M)$ denote the canonical immersion, which is continuous due to Theorem \ref{basictopolgyunfirom}. Thus, we can define a new Borel measure
\begin{equation}\label{defwideeta}
\eta:= \widetilde{\eta}\circ T\in \mathscr{P}(C([0,1];M)),
\end{equation}
which is concentrated on $\AC^p([0,1];M)$,
i.e., Statement \eqref{supent1} follows.

\medskip

\noindent \textbf{Step 5.} In this step, we prove Statement \eqref{supent2}.
It is sufficient to show that for every $t\in [0,1]$,
\begin{equation}\label{lastidentityee}
\int_{C([0,1];M)}\varphi(e_t(\gamma)){\dd}\eta(\gamma)=\int_{M}\varphi(x){\dd}\mu_t(x),\quad \forall\,\varphi\in \Lip(M)\cap C_b(M),
\end{equation}
where $\Lip(M):=\left\{f\in C(M)\,|\ \exists \ C>0 \text{ such that }f(y)-f(x)\leq C\,d_F(x,y),\ \forall\,x,y\in M\right\}$.

In fact, if (\ref{lastidentityee}) holds, then for any $\varphi\in C_b(M)$, define a sequence $\varphi_k\in \Lip(M)\cap C_b(M)$ by
\[
\varphi_k(x):=\sup_{y\in M}\left[ \varphi(y)-k \,d_F(x,y) \right].
\]
Clearly, $\inf \varphi\leq \varphi(x)\leq \varphi_k(x)\leq \sup \varphi$ and $\lim_{k\rightarrow \infty}\varphi_k(x)=\varphi(x)$. Thus, \eqref{lastidentityee} (for $\varphi_k$) combined with the dominated convergence theorem yields
\begin{equation}
\int_{M}\varphi(x){\dd}(e_t)_\sharp\eta(x)=\int_{C([0,1];M)}\varphi(e_t(\gamma)){\dd}\eta(\gamma)=\int_{M}\varphi(x){\dd}\mu_t(x), \quad \forall t\in [0,1],
\end{equation}
which indicates $(e_t)_\sharp \eta=\mu_t$, i.e., Statement \eqref{supent2} follows.

In the sequel, we show (\ref{lastidentityee}). Given any $\varphi\in \Lip(M)\cap C_b(M)$, by dividing a constant, we may assume
\[
\varphi\in \Lip_1(M)\cap C_b(M):=\{f\in C_b(M)\,|\ f(y)-f(x)\leq d_F(x,y),\ \forall\,x,y\in X\}.
\]
Set
$g(t):=\int_{M} \varphi(x){\dd}\mu_t(x)$.
Firstly, we claim that $g(t)$ is uniformly continuous. In fact, for any $s\leq t$, owing  to Theorems \ref{wassdistacbasisthe} $\&$ \ref{reversbilityofW}, we have
\begin{align*}
g(t)-g(s)&= \int_{M} \varphi{\dd}\mu_t-\int_{M} \varphi{\dd}\mu_s\leq \sup_{\psi\in \Lip_1(M)}\left( \int_{M} \psi{\dd}\mu_t-\int_{M} \psi{\dd}\mu_s  \right)\notag=W_1(\mu_s,\mu_t)\leq W_q(\mu_s,\mu_t),\notag\\
g(s)-g(t)&\leq W_1(\mu_t,\mu_s)\leq W_q(\mu_t,\mu_s),
\end{align*}
which together with \eqref{facinwpspace2} yield the uniform continuity of $g$.

In view of \eqref{defti}, define a sequence of piecewise constant functions
\[
g_N(t):=g(t^i)=\int_{M} \varphi(x){\dd}\mu_{t^i}(x),\quad \text{if }t\in [t^i,t^{i+1}),
\]
which converges uniformly to $g$ in $[0,1]$ as $N\rightarrow \infty$ due to the uniform continuity. Thus, for every test function $\zeta\in C_b([0,1])$, we have
\begin{equation}\label{fistzetalimit}
\lim_{N\rightarrow \infty}\int^1_0 \zeta(t)g_N(t){\dd}t=\int^1_0 \zeta(t)g(t){\dd}t=\int^1_0\zeta(t)\left(\int_{M} \varphi(x){\dd}\mu_t(x)\right){\dd}t.
\end{equation}
On the other hand, the  proof of (\ref{etaNtNconverse}) together with Fubini's theorem yields
\begin{align}\label{secondlimit}
&\int^1_0 \zeta(t) g_N(t){\dd}t=\sum_{i=0}^{2^{N}-1}\int^{t^{i+1}}_{t^i}\zeta(t)\left(\int_M      \varphi(x){\dd}\mu_{t^i}(x)\right){\dd}t\notag\\
=&\int^1_0 \zeta(t) \left(\int_{\mathcal {M}([0,1];M)}\varphi(e_t(\gamma)){\dd}\widetilde{\eta}_N(\gamma)\right){\dd}t=\int_{\mathcal {M}([0,1];M)}\left(\int^1_0\zeta(t)\varphi(e_t(\gamma)){\dd}t\right){\dd}\widetilde{\eta}_N(\gamma).
\end{align}

Since $\varphi$ is bounded, the map
\[
\gamma\in \mathcal {M}([0,1];M)\longmapsto \int^1_0\zeta(t)\varphi(e_t(\gamma)){\dd}t\in \mathbb{R}
\]
is    bounded and continuous (by a contradiction argument about subsequence). The narrow convergence  $\widetilde{\eta}_{N_k}\Rightarrow\widetilde{\eta}$ in $\mathscr{P}(\mathcal {M}([0,1];M))$ (see the beginning of Step 3) combined with Fubini's theorem and \eqref{defwideeta}  yields
\begin{align*}\label{lastlimi}
&\lim_{k\rightarrow \infty}\int_{\mathcal {M}([0,1];M)}\left(\int^1_0 \zeta(t)\varphi(e_t(\gamma)){\dd}t\right){\dd}\widetilde{\eta}_{N_k}(\gamma)=\int_{\mathcal {M}([0,1];M)}\left(\int^1_0 \zeta(t)\varphi(e_t(\gamma)){\dd}t\right){\dd}\widetilde{\eta}(\gamma)\\
&=\int_{C([0,1];M)}\int^1_0\zeta(t)\varphi(e_t(\gamma)){\dd}t{\dd}\eta(\gamma)=\int^1_0 \zeta(t) \left( \int_{C([0,1];M)} \varphi(e_t(\gamma)){\dd}\eta(\gamma)\right){\dd}t,
\end{align*}
which together with \eqref{secondlimit}  furnishes
\[
\lim_{k\rightarrow \infty}\int^1_0 \zeta(t) g_{N_k}(t){\dd}t=\int^1_0 \zeta(t) \left(\int_{C([0,1];M)} \varphi(e_t(\gamma)){\dd}\eta(\gamma)\right){\dd}t.
\]
This combined with (\ref{fistzetalimit}) yields
\[
\int^1_0\zeta(t)\left(\int_{M} \varphi(x){\dd}\mu_t(x)\right){\dd}t=\int^1_0 \zeta(t) \left(\int_{C([0,1];M)} \varphi(e_t(\gamma)){\dd}\eta(\gamma)\right){\dd}t,
\]
which indicates that for $\mathscr{L}^1$-a.e. $t\in (0,1)$,
\begin{equation}\label{legeequalmeasure}
\int_{C([0,1];M)} \varphi(e_t(\gamma)){\dd}\eta(\gamma)=\int_{M} \varphi(x){\dd}\mu_t(x).
\end{equation}
Note that both $t\mapsto \int_{M} \varphi(x){\dd}\mu_t(x)$ and $t\mapsto \int_{C([0,1];M)} \varphi(e_t(\gamma)){\dd}\eta(\gamma)$ are continuous because  $t \mapsto \mu_t $ and $t\mapsto (e_t)_\sharp \eta$ are narrowly continuous in $ \mathscr{P}(M)$. Hence, \eqref{legeequalmeasure} holds for all $t\in [0,1]$, i.e., (\ref{lastidentityee}) is true.

\medskip

\noindent \textbf{Step 6.} In this step, we prove Statement \eqref{supent3}.

First, we claim that for all $s_1,s_2\in [0,1]$ with $s_1<s_2$,
\begin{equation}\label{s1s2+hestim}
\int_{\mathcal {M}([0,1];M)}\int^{s_2}_{s_1}\left(\frac{{d}_F(\gamma(t),\gamma(t+h)) }{h }\right)^p{\dd}t{\dd}\widetilde{\eta}(\gamma)\leq  \int^{s_2+h}_{s_1}|\mu'_+|_p^p(t){\dd}t,
\end{equation}
for every $h\in (0,1-s_2)$.

In fact, for any $h\in (0,1-s_2)$, choose $N\in \mathbb{N}$ and  $k\geq 1$ such that $1/2^N\leq h$ and (\ref{khcontroll}) holds. Setting
\[
s^N_1:=\frac{\underline{i}}{2^N}:=\max\left\{ \frac{i}{2^N}\,\Big|\ \frac{i}{2^N}\leq s_1 \right\},\quad s^N_2:=\frac{\bar{i}}{2^N}:=\min\left\{ \frac{i}{2^N}\,\Big|\ \frac{i}{2^N}\geq s_2 \right\},
\]
and reasoning as in the proof of (\ref{khNinequality}) we obtain
\[
\int^{s_2}_{s_1}d_F(\sigma_t(\mathbf{x}),\sigma_{t+h}(\mathbf{x}))^p{\dd}t\leq \frac{(k+1)^p}{2^N}\sum_{j=\underline{i}}^{\bar{i}-1}d_F(x_j,x_{j+1})^p,
\]
which together with (\ref{khcontroll}) yields
\begin{equation}\label{dfsing}
\int^{s_2}_{s_1}d_F(\sigma_t(\mathbf{x}),\sigma_{t+h}(\mathbf{x}))^p{\dd}t\leq h^p \left(  \frac{k+1}{k} \right)^p 2^{N(p-1)}\sum_{j=\underline{i}}^{\bar{i}-1}d_F(x_j,x_{j+1})^p.
\end{equation}
Moreover, \eqref{forprobb22} combined with Theorem \ref{vilocitythem} furnishes
\begin{equation}\label{sigmalable1}
\int_{\mathbf{M}}d_F(x_j,x_{j+1})^p{\dd}\pi_N(\mathbf{x})= W_p(\mu_{t^j},\mu_{t^{j+1}})^p\leq \frac{1}{2^{N(p-1)}}\int^{t^{j+1}}_{t^j}|\mu'_+|_p^p(t){\dd}t.
\end{equation}
A similar but easier argument to that of \eqref{keylatboelone} together with \eqref{dfsing} and \eqref{sigmalable1} yields
\begin{equation*}\label{lpestmiate1}
\int_{\mathcal {M}([0,1];M)}\int^{s_2}_{s_1}\left(\frac{{d}_F(\gamma(t),\gamma(t+h)) }{h }\right)^p{\dd}t{\dd}\widetilde{\eta}_N(\gamma)\leq \left( \frac{k+1}{k} \right)^p \int_{s^N_1}^{s^N_2+h}|\mu'_+|_p^p(t){\dd}t.
\end{equation*}
Now by letting $N\rightarrow \infty$ and hence, $k\rightarrow \infty$, we obtain (\ref{s1s2+hestim}) from the above inequality.

Secondly, by \eqref{defwideeta}  we have
 \begin{align*}
\int_{\mathcal {M}([0,1];M)}\int^{s_2}_{s_1}\left(\frac{d_F(\gamma(t),\gamma(t+h)) }{h }\right)^p{\dd}t{\dd}\widetilde{\eta}(\gamma)=\int_{C([0,1];M)}\int^{s_2}_{s_1}\left(\frac{d_F(\tilde{\gamma}(t),\tilde{\gamma}(t+h))}{h}\right)^p{\dd}t{\dd}\eta(\tilde{\gamma}),
\end{align*}
which together with  (\ref{s1s2+hestim}) furnishes
\begin{align*}
\int_{C([0,1];M)}\int^{s_2}_{s_1}\left(\frac{d_F(\tilde{\gamma}(t),\tilde{\gamma}(t+h)) }{h }\right)^p{\dd}t{\dd}\eta(\tilde{\gamma})\leq  \int^{s_2+h}_{s_1}|\mu'_+|_p^p(t){\dd}t.
\end{align*}
Since $\eta$ is concentrated on $\AC^p([0,1];M)$, by letting $h\rightarrow 0^+$,  Fatou's lemma yields
\begin{align}\label{energeyidenetii}
\int_{C([0,1];M)}\int^{s_2}_{s_1}{F^p(\gamma'(t))}{\dd}t{\dd}\eta(\gamma)\leq  \int^{s_2}_{s_1}|\mu'_+|_p^p(t){\dd}t
\end{align}
for every $s_1,s_2\in [0,1]$ such that $s_1<s_2$. Thus, it follows from \eqref{energeyidenetii}, Fubini's theorem and the Lebesgue differentiation theorem that
\begin{equation}\label{bascimulargeerinf}
 |\mu'_+|^p_p(t)\geq  \int_{C([0,1];M)} F^p(\gamma'(t)) \,{\dd}\eta(\gamma) \ \text{ for $\mathscr{L}^1$-a.e. $t\in (0,1)$}.
\end{equation}

In order to show that reverse of \eqref{bascimulargeerinf}, choose $t\in (0,1)$ such that $|\mu'_+|_p(t)$ exists.
Given $h>0$, set $\pi_{t,t+h}:=(e_t,e_{t+h})_\sharp\eta\in \Pi(\mu_t,\mu_{t+h})$. By Fatou's lemma and \eqref{supent1}, we have
\begin{align}
&|\mu'_+|^p_p(t)=\lim_{t\rightarrow 0^+}\left(\frac{W_p(\mu_t,\mu_{t+h}) }{h}\right)^p\leq \limsup_{h\rightarrow 0^+}\int_{M\times M}\left(\frac{d_F(x,y)}{h}\right)^p{\dd}\pi_{t,t+h}(x,y)\notag\\
\leq&  \limsup_{h\rightarrow 0^+}\int_{C([0,1];M)}\left(  \frac{d_F(e_t(\gamma),e_{t+h}(\gamma))}{ h} \right)^p{\dd}\eta(\gamma)\leq \int_{C([0,1];M)}\limsup_{h\rightarrow 0^+}\left(  \frac{d_F(\gamma(t),\gamma(t+h))}{h} \right)^p{\dd}\eta(\gamma)\notag\\
=&\int_{C([0,1];M)} F(\gamma'(t))^p{\dd}\eta(\gamma),\label{speectonr}
\end{align}
which together with \eqref{bascimulargeerinf} furnishes  Statement \eqref{supent3}.
\end{proof}

Proceeding as in the above proof, one can get a similar structure theorem of $\mu_t\in \BAC^p([0,1];\mathscr{P}_p(M))$.
In particular, a stronger result reads as follows.

\begin{theorem}\label{basicfinterevers}
Let $(M,F)$ be a forward complete Finsler manifold and $p\in (1,\infty)$. Thus,
for any $\mu_t\in  \AC^p([0,1];\mathscr{P}_p(M))$, there exist  $\eta_\pm\in \mathscr{P}(C([0,1];M))$ such that
\begin{enumerate}[{\rm (i)}]

\item \label{dousupent1} $\eta_\pm$ are concentrated on $\AC^p([0,1];M)$;

\smallskip

\item\label{dousupent2} $\mu_t=(e_t)_\sharp \eta_+=(e_t)_\sharp \eta_-$ for any $t\in [0,1]$;
\item\label{dousupent3}  $\ds |\mu'_\pm|^p_p(t)= \int_{C([0,1];M)} F^p(\pm\gamma'(t)) \,{\dd}\eta_\pm(\gamma)$ for $\mathscr{L}^1$-a.e. $t\in (0,1)$.
\end{enumerate}
In particular, for $\mathscr{L}^1$-a.e. $t\in (0,1)$,
\begin{align}\label{speedcontroll}
 |\mu'_+|_p^p(t)\leq \int_{C([0,1];M)}F^p(\gamma'(t)){\dd}\eta_-(\gamma), \quad
 |\mu'_-|_p^p(t)\leq \int_{C([0,1];M)}F^p(-\gamma'(t)){\dd}\eta_+(\gamma).
\end{align}
\end{theorem}
\begin{proof}
Note that $\mu_t\in \AC^p([0,1];\mathscr{P}_p(M))$ is continuous in $(\mathscr{P}_p(M),\hat{W}_p)$ and hence, $\mathscr{A}:=\{\mu_t\,|\,t\in [0,1]\}$ is compact in $\mathscr{P}(M)$. In particular, \eqref{facinwpspace2} is vail for $q=p$.
Hence, by repeating the same proof of Theorem \ref{maintheorem}, one can show  the existence of $\eta_+\in \mathscr{P}(C([0,1];M))$ which satisfies \eqref{dousupent1}--\eqref{dousupent3}.  The $\eta_-$-case follows by considering the reverse Finsler manifold and using the same argument.
And an easy modification of \eqref{speectonr} combined with \eqref{dousupent2} yields \eqref{speedcontroll}.
\end{proof}

\begin{remark}\label{etapmexplain}    It is clear that $\eta_-$ is the counterpart of $\eta_+$ in the reverse Finsler manifold.
Note that we construct  $\eta_+$ (resp., $\eta_-$) by the optimal transference plans with respect to $W_p$  (resp., $\overleftarrow{W_p}$). Hence, for a  reversible Finsler manifold, we have $\eta_+=\eta_-$.
\end{remark}

\subsection{Continuity equations}\label{Contineqa} Continuity equations play a key role in the study of diffusion equations (cf. \cite{AGS,Vi}).
As an application of Theorem \ref{basicfinterevers},
  we investigate continuity equations in the non-compact Finsler case. Also refer to \cite{OS,OZ} for the compact Finsler case, \cite{BERN,Erb} for the Riemannian case and \cite{Li} for the Banach case.

\begin{definition} Let $(M,F)$ be a  Finsler manifold and let $p\in (1,\infty)$.
Given a narrowly continuous curve $\mu_t\in C([0,1]; \mathscr{P}(M))$,
\begin{enumerate}[{\rm (i)}]

\item \label{speicalmeaure} a   measure $\bar{\mu}\in \mathscr{P}([0,1]\times M)$ is said to be {\it associated to $\mu_t$} if for every bounded Borel function $\varphi:[0,1]\times M\rightarrow \mathbb{R}$,
\begin{equation}\label{baretadef}
\int_{[0,1]\times M}\varphi(t,x){\dd}\bar{\mu}(t,x)=\int^1_0 \int_{M}\varphi(t,x){\dd}\mu_t(x){\dd}t.
\end{equation}
In the subsection, we always use $\bar{\mu}$ to denote the measure  associated to $\mu_t$.

\smallskip

\item \label{vectorfieldnew}  a time-dependent Borel vector field $\Bv:[0,1]\times M\rightarrow TM$
is said to  {\it belong to $L^p(\bar{\mu};TM_+)$} (resp., $L^p(\bar{\mu};{TM_-})$) if   $\Bv_t(x)\in T_xM$ for $\bar{\mu}$-a.e.  $(t,x)\in [0,1]\times M$ and
\[
 \int_{[0,1]\times M} F^p(\Bv_t(x)){\dd}\bar{\mu}(t,x)<\infty \quad \left( \text{ resp., }  \int_{[0,1]\times M} {F}^p(-\Bv_t(x)){\dd}\bar{\mu}(t,x)<\infty\right);
\]
It is remarkable that $L^p(\bar{\mu};TM_+)$ may be different from  $L^p(\bar{\mu};{TM_-})$ when
 the reversibility is infinite, in which case neither of them is a vector space (cf. \cite{KR}).

\smallskip

\item \label{conteqation} given $\Bv\in L^p(\bar{\mu};TM_+)\cup L^p(\bar{\mu};{TM}_-)$, the pair $(\mu_t,\Bv)$ is said to {\it satisfy the continuity equation}
\begin{equation}\label{contequa}
\partial_t\mu_t+\di(\Bv_t\mu_t)=0
\end{equation}
if there holds
\[
\frac{{\dd}}{{\dd}t}\int_{M}\phi(x){\dd}\mu_t(x)=\int_{M}\langle \Bv_t, {\dd}\phi\rangle {\dd}\mu_t,\quad \forall\phi\in C^1_0(M),
\]
where the equality is intended in the sense of distribution in $(0,1)$.

\end{enumerate}

\end{definition}

\begin{theorem}\label{contequat1}Let $(M,F)$ be a forward complete Finsler manifold and let $p\in (1,\infty)$.
Given $\mu_t\in \AC^p([0,1];\mathscr{P}_p(M))$,  there exists two vector fields $\Bv^\pm\in L^p(\bar{\mu};TM_\pm)$ such that  both $(\mu_t,\Bv^\pm)$ satisfy the continuity equation \eqref{contequa} and
\begin{align}\label{normcontrollvvt}
\| {\Bv^+_t}\|_{L^p(\mu_t;M_+)}\leq |\mu'_+|_p(t),\quad \| {{\Bv}^-_t}\|_{L^p(\mu_t;{M}_-)}\leq |\mu'_-|_p(t),
\end{align}
for $\mathscr{L}^1$-a.e. $t\in (0,1)$, where
\[
\| \cdot\|_{L^p(\mu_t;M_+)}:=\left(\int_{M}F^p(\cdot){\dd}\mu_t\right)^{1/p},\quad \| \cdot\|_{L^p(\mu_t;{M}_-)}:=\left(\int_{M}{F}^p(-~\cdot~){\dd}\mu_t\right)^{1/p}.
\]
 \end{theorem}

\begin{proof}We focus on the ``$+$'' case since the ``$-$" case can be derived from a similar argument.
By  Theorem \ref{basicfinterevers}, there exists $\eta:=\eta_+\in \mathscr{P}(C([0,1];M))$ such that $\mu_t=(e_t)_\sharp\eta$. Now define a measure $\breve{\eta}\in \mathscr{P}([0,1]\times C([0,1];M))$ and an evaluation map $\e:[0,1]\times C([0,1];M)\rightarrow [0,1]\times M$ as
\[
\breve{\eta}:=\mathscr{L}^1|_{[0,1]}\otimes \eta,\quad \e(t,\gamma):=(t,e_t(\gamma)).
\]
 Thus, owing to \eqref{baretadef}, it is easy to check  $\e_\sharp \breve{\eta}=\bar{\mu}$.

According to \cite[Theorem 5.3.1]{AGS}, the disintegration of $\breve{\eta}$ with respect to $\e$ yields a family of Borel probability measures $\breve{\eta}_{t,x}\in \mathscr{P}(C([0,1];M))$ concentrated on $\{\gamma\in C([0,1];M)\,|\  e_t(\gamma)=x \}=(e_t)^{-1}(x)$ such that for every $\varphi:[0,1]\times C([0,1];M)\rightarrow \mathbb{R} $ with $\varphi\in L^1(\breve{\eta})$,
\begin{align}
\gamma&\quad\longmapsto\quad \varphi(t,\gamma)\in L^1(\breve{\eta}_{t,x})\ \text{ for $\bar{\mu}$-a.e. $(t,x)\in [0,1]\times M$},\label{lietatx1}\\
(t,x)&\quad\longmapsto\quad \int_{C([0,1];M)}\varphi(t,\gamma){\dd}\breve{\eta}_{t,x}(\gamma)\in L^1(\bar{\mu}),\label{lietatx2}\\
\int_{[0,1]\times C([0,1];M)}\varphi(t,\gamma){\dd}\breve{\eta}(t,\gamma)&\quad=\quad\int_{[0,1]\times M}\int_{C([0,1];M)} \varphi(t,\gamma){\dd}\breve{\eta}_{t,x}(\gamma){\dd}\bar{\mu}(t,x),\label{disintegr3}
\end{align}
and the measures $\breve{\eta}_{t,x}$ are uniquely determined for $\bar{\mu}$-a.e. $(t,x)\in [0,1]\times M$.

 Define a set
\[
\mathcal {A}:=\{(t,\gamma)\in [0,1]\times C([0,1];M)\,|\ {\gamma'}(t) \text{ exists} \}.
\]
Thus, $\mathcal {A}$ is a Borel subset of $[0,1]\times  C([0,1];M)$ as the map $(t,\gamma)\mapsto d_F(\gamma(t),\gamma(t+h))/h$ are continuous from $[0,1]\times C([0,1];M)$ to $\mathbb{R}$ for every $h>0$. Since $\eta$ is concentrated on $\AC^p([0,1];M)$, we have
\[
\mathscr{L}^1\left( \left\{ t\in [0,1]\,|\  (t,\gamma)\in \mathcal {A}^c    \right\}  \right)=0, \quad  \text{for  $\eta$-a.e. $\gamma\in C([0,1];M)$}.
\]
Now  Fubini's theorem yields
\begin{align*}
0=\int_{C([0,1];M)} \mathscr{L}^1\left( \left\{ t\in [0,1]\,|\  (t,\gamma)\in \mathcal {A}^c   \right\}  \right){\dd}\eta(\gamma)=\int_{\mathcal {A}^c}{\dd}\eta{\dd}t=\breve{\eta}(\mathcal {A}^c).
\end{align*}
Hence, we can define a map $\Psi:[0,1]\times C([0,1];M)\rightarrow TM$ as
$\Psi(t,\gamma):={\gamma'}(t)$, which is well defined for $\breve{\eta}$-a.e. $(t,\gamma)\in [0,1]\times C([0,1];M)$.

Now we claim that $F\circ \Psi\in L^p(\breve{\eta})$ and $F\circ\Psi(t,\cdot)\in L^1(\breve{\eta}_{t,x})$ for $\bar{\mu}$-a.e. $(t,x)\in [0,1]\times M$. Factually, since $\mu_t\in \AC^p([0,1];\mathscr{P}_p(M))$, Theorem \ref{basicfinterevers} furnishes
\begin{align}
\int_{[0,1]\times C([0,1];M)} F^p\circ\Psi(t,\gamma) {\dd}\breve{\eta}(t,\gamma)
=\int^1_0 \left( \int_{C([0,1];M)}F^p(\gamma'(t)){\dd}\eta(\gamma)  \right){\dd}t=\int^1_0 |\mu'_+|^p_p(t){\dd}t<\infty,\label{pfpcontoll}
\end{align}
which indicates $F\circ \Psi\in L^p(\breve{\eta})$. The H\"older inequality yields $F\circ \Psi\in L^1(\breve{\eta})$  
and hence,
 \eqref{disintegr3} implies
\begin{align*}
\int_{[0,1]\times M}\int_{C([0,1];M)} F\circ\Psi (t,\gamma){\dd}\breve{\eta}_{t,x}(\gamma){\dd}\bar{\mu}(t,x)=\int_{[0,1]\times C([0,1];M)}F\circ\Psi (t,\gamma){\dd}\breve{\eta}(t,\gamma)<\infty,
\end{align*}
which means $F\circ\Psi(t,\cdot)\in L^1(\breve{\eta}_{t,x})$ for $\bar{\mu}$-a.e. $(t,x)\in [0,1]\times M$. So the claim is true.

Now we show that
 for $\bar{\mu}$-a.e. $(t,x)\in [0,1]\times M$, the vector field
\begin{equation}\label{definbv}
{\Bv^+_t}(x):=\int_{C([0,1];M)}\Psi(t,\gamma) {\dd}\breve{\eta}_{t,x}(\gamma)=\int_{C([0,1];M)} {\gamma'}(t) {\dd}\breve{\eta}_{t,x}(\gamma)
\end{equation}
is well defined and particularly $\Bv^+\in L^p(\bar{\mu};TM_+)$. Indeed, it follows by $\breve{\eta}(\mathcal {A}^c)=0$ and \eqref{disintegr3} that
\[
\breve{\eta}_{t,x}(\{\gamma\in C([0,1];M)\,|\ (t,\gamma)\in \mathcal {A}^c\})=0, \quad \text{for $\bar{\mu}$-a.e. $(t,x)\in [0,1]\times M$},
\]
which implies that
$\breve{\eta}_{t,x}$ is concentrated on
\[
\mathscr{U}_{t,x}:=\{\gamma\in C([0,1];M)\,|\ \gamma(t)=x,\ \gamma'(t) \text{ exists}\}\subset (e_t)^{-1}(x).
\]
Thus, $\Psi(t,\gamma)=\gamma'(t)\in T_xM$ for every $\gamma\in \mathscr{U}_{t,x}$. Moreover, given $\zeta\in T^*_xM$,
 by \eqref{tangninnerprot},  Jensen's inequality and  $F\circ\Psi(t,\cdot)\in L^1(\breve{\eta}_{t,x})$, we have
\begin{align*}
\langle {\Bv^+_t}(x),\zeta  \rangle=\left\langle\int_{C([0,1];M)}\Psi(t,\gamma) {\dd}\breve{\eta}_{t,x}(\gamma),\zeta\right\rangle\leq F^*(\zeta)\int_{{C([0,1];M)}}F\circ\Psi(t,\gamma) {\dd}\breve{\eta}_{t,x}(\gamma)<\infty,
\end{align*}
which together with Pettis' measurability theorem (cf. \cite{HKST}) implies that $\Psi(t,\cdot)$ is $\breve{\eta}_{t,x}$-measurable and hence,   ${\Bv^+_t}(x)\in T_xM$ is well defined. Moreover,
the convexity of $F^p$ combined with \eqref{pfpcontoll} and \eqref{disintegr3} yields
\begin{align}\label{contrallestvv}
&\int_{[0,1]\times M} F^p({\Bv^+_t}(x)){\dd}\bar{\mu}(t,x)=\int_{[0,1]\times M} F^p\left(  \int_{C([0,1];M)} {\gamma'}(t) {\dd}\breve{\eta}_{t,x}(\gamma)   \right){\dd}\bar{\mu}(t,x)\notag\\
\leq & \int_{[0,1]\times M}   \int_{C([0,1];M)} F^p({\gamma'}(t)) {\dd}\breve{\eta}_{t,x}(\gamma)   {\dd}\bar{\mu}(t,x)=\int_{[0,1]\times C([0,1];M)} F^p({\gamma'}(t)){\dd}\breve{\eta}(t,\gamma)<\infty,
\end{align}
which implies  $\Bv^+\in L^p(\bar{\mu};TM_+)$.

Now we show \eqref{normcontrollvvt}. In fact, for any $[a,b]\subset [0,1]$,  the definition of $\Bv^+$, \eqref{contrallestvv} and \eqref{pfpcontoll} furnish
\begin{align*}
&\int^b_a\| {\Bv^+_t}\|^p_{L^p(\mu_t;M_+)}{\dd}t=\int^b_a \int_{M}F^p( {\Bv^+_t}(x)){\dd}\mu_t(x){\dd}t=\int_{[0,1]\times M}\mathbf{1}_{[a,b]}(t) \,F^p( {\Bv^+_t}(x)){\dd}\bar{\mu}(t,x)\\
&\leq \int_{[0,1]\times M}\mathbf{1}_{[a,b]}(t) F^p( {\gamma'}(t)) {\dd}\breve{\eta}(t,\gamma)=\int^b_a |\mu'_+|_p^p(t){\dd}t.
\end{align*}

It remains to show that $(\mu_t,\Bv^+)$ satisfies \eqref{contequa}. First we claim that     $t\mapsto \int_{M}\phi{\dd}\mu_t$ is absolutely continuous for every $\phi\in C^1_0(M)$. In fact,  an  argument  similar to \eqref{abcontofphi} yields
\[
\sup_{x,y\in M}\frac{|\phi(y)-\phi(x)|}{d_F(x,y)}\leq \sup_{z\in M}\max\{ F^*(\pm {\dd}\phi(z))\}=:C(\phi)<\infty.
\]
Thus,
for every $[a,b]\subset [0,1]$, by choosing an optimal transference plan $\pi_{a,b}$ from $\mu_a$ to $\mu_b$ with respect to $W_1$, we have
\begin{align*}
&\left| \int_{M}\phi{\dd}\mu_b-\int_{M}\phi{\dd}\mu_a  \right|\leq \int_{M\times M} |\phi(y)-\phi(x)|{\dd}\pi_{a,b}(x,y)=\int_{M\times M} \frac{|\phi(y)-\phi(x)|}{d_F(x,y)}d_F(x,y){\dd}\pi_{a,b}(x,y)\\
&\leq C(\phi) \int_{M\times M}d_F(x,y){\dd}\pi_{a,b}(x,y)\leq  C(\phi)\,W_p(\mu_a,\mu_b)\leq C(\phi) \int^b_a |\mu'_+|_p(t){\dd}t,
\end{align*}
which implies the absolute continuity of $t\mapsto \int_{M}\phi{\dd}\mu_t$ due to Theorem \ref{vilocitythem}. Now
Theorem \ref{basicfinterevers} together with Proposition \ref{smoothisabsoltey} and \eqref{disintegr3} yields
\begin{align*}
&\int^{b}_{a} \left(\frac{\dd}{{\dd}t}\int_M\phi\mu_t\right){\dd}t=\int_{M}\phi{\dd}\mu_{b}-\int_{M}\phi{\dd}\mu_{a}=\int_{C([0,1];M)} \Big(\phi(e_b(\gamma))-\phi(e_a(\gamma))\Big){\dd}\eta(\gamma)\\
=&\int_{C([0,1];M)} \left(\int^{b}_{a} \frac{\dd}{{\dd} t}\phi(\gamma(t)) {\dd}t\right){\dd}\eta(\gamma)=\int_{[0,1]\times C([0,1];M) } \mathbf{1}_{[a,b]}(t)\,\langle  \Psi(t,\gamma),{\dd}\phi\rangle\,{\dd}\breve{\eta}(t,\gamma)\\
=&\int_{[0,1]\times M} \int_{C([0,1];M)}\mathbf{1}_{[a,b]}(t) \langle  \Psi(t,\gamma),{\dd}\phi\rangle {\dd}\breve{\eta}_{t,x}(\gamma){\dd}\bar{\mu}(t,x)=\int_{[0,1]\times M}\mathbf{1}_{[a,b]}(t)\langle {\Bv^+_t},{\dd}\phi\rangle {\dd}\bar{\mu}(t,x)\\
=&\int^{b}_{a}\left(\int_M\langle {\Bv^+_t},{\dd}\phi\rangle {\dd}\mu_t\right){\dd}t,
\end{align*}
which  concludes the proof.
\end{proof}

Now we discuss the reverse of Theorem \ref{contequat1}. Recall that
   the {\it uniform constant} $\Lambda_F(M)$  of $(M,F)$ (cf. \cite{E}) is defined as
\[
 \Lambda_F(M):=\underset{y,v,z\in TM\backslash\{0\}}{\sup}\frac{g_v(y,y)}{g_z(y,y)},
\]
where $g_v(y,y):=g_{ij}(v)y^iy^j$.
Clearly, ${\Lambda_F}(M)\geq 1$ with equality if and
only if $F$ is Riemannian.

\begin{proposition}\label{uncontcounequa1}
Let $(M,F)$ be a forward complete Finsler manifold with finite uniform constant and let $p\in (1,\infty)$. If $\mu_t\in C([0,1];\mathscr{P}(M))$ satisfies \eqref{contequa} for some vector field $\Bv\in L^p(\bar{\mu};TM_+)$,
then $\mu_t\in \AC^p([0,1];\mathscr{P}_p(M))$ and for $\mathscr{L}^1$-a.e. $t\in (0,1)$,
\begin{equation*}\label{coneqbasicproper}
 |\mu'_+|_p(t)\leq \|{\Bv}_t\|_{L^p(\mu_t;M_+)}, \quad |\mu'_-|_p(t)\leq \|{\Bv}_t\|_{L^p(\mu_t;M_-)}.
\end{equation*}

\end{proposition}
\begin{proof}Owing to $\Lambda_F(M)<\infty$, one can define a Riemannian metric $\hat{g}$ on $M$ by
\[
\hat{g}_x(X,Y):=\frac{1}{\nu_x(S_xM)}\int_{S_xM}g_{y}(X,Y){\dd}\nu_x(y),\quad \forall\,X,Y\in T_xM,\ \forall\,x\in M,
\]
where $\nu_x$ is the Riemannian measure induced by $g$ on $S_xM:=\{y\in T_xM\,|\ F(x,y)=1\}$. Clearly,
\begin{equation}\label{equivmetricrIE}
{\Lambda_F(M)^{-1}} {\hat{g}(X,X)}\leq F(X)^2\leq {\Lambda_F(M)}\,{\hat{g}(X,X)},\quad \forall\,X\in TM,
\end{equation}
which means that $(M,\hat{g})$ is a complete Riemannian manifold. Moreover,  $\Bv\in L^p(\bar{\mu};TM_+)$ together with \eqref{equivmetricrIE} and the H\"older inequality implies
\[
\int_{[0,1]\times M} \sqrt{\hat{g}({\Bv_t}(x),{\Bv_t}(x))}  {\dd}\bar{\mu}(t,x)\leq     {\Lambda_F(M)}^{\frac{1}2} \left(\int_{[0,1]\times M} F^p({\Bv_t}(x)){\dd}\bar{\mu}(t,x)\right)^{\frac1p}<\infty.
\]
Thus,
owing to \cite[Theorem 5.8]{BERN}, there exists a   measure $\eta\in \mathscr{P}(C([0,1];M))$ such that $(e_t)_\sharp \eta=\mu_t$ for all $t\in [0,1]$ and $\eta$ is concentrated on the set of
curves $\gamma$ solving ${\gamma}'(t)={\Bv}_t(\gamma(t))$, which are differentiable at $\mathscr{L}^1$-a.e. $t\in (0,1)$. Hence, for any $[a,b]\subset [0,1]$, since $(e_a,e_b)_\sharp \eta$ is a transference plan from $\mu_a$ to $\mu_b$,   the H\"older inequality combined with $\mathfrak{e}_\sharp\breve{\eta}=\bar{\mu}$ yields
\begin{align*}
&W_p(\mu_a,\mu_b)^p\leq \int_{M\times M}d_F(x,y)^p{\dd} (e_a,e_b)_\sharp \eta=    \int_{C([0,1];M)}d_F(\gamma(a),\gamma(b))^p\,{\dd}\eta(\gamma)\\
&\leq \int_{C([0,1];M)} \left( \int^b_a F(\gamma'(t)){\dd}t  \right)^p{\dd}\eta(\gamma)=\int_{C([0,1];M)} \left( \int^b_a F({\Bv_t}(\gamma(t))){\dd}t  \right)^p{\dd}\eta(\gamma)\\
&\leq  (b-a)^{p-1} \int_{C([0,1];M)}\int^b_a F^p({\Bv_t}(e_t(\gamma))){\dd}t{\dd}\eta(\gamma)=(b-a)^{p-1} \int^b_a \int_M F^p({\Bv_t}){\dd}\mu_t{\dd}t,
\end{align*}
which implies $|\mu'_+|_p(t)\leq \|{\Bv}_t\|_{L^p(\mu_t;M_+)}$  for $\mathscr{L}^1$-a.e. $t\in (0,1)$ and hence, $\mu_t\in \FAC^p([0,1];\mathscr{P}_p(M))$. One can conclude the proof by considering the reverse Finsler manifold.
\end{proof}

\begin{corollary}\label{unformuqncon11}
Let $(M,F)$ be a forward complete Finsler manifold with finite uniform constant and let $p\in (1,\infty)$. For every $\mu_t\in \AC^p([0,1];\mathscr{P}_p(M))$, there exist  two vector fields $\Bv^\pm\in L^p(\bar{\mu};TM_\pm)$ such that $(\mu_t,\Bv^\pm)$ satisfy the continuity equation and for $\mathscr{L}^1$-a.e. $t\in (0,1)$,
  \begin{align}\label{constaid}
    \| \Bv^+_t\|_{L^p(\mu_t;M_+)}= |\mu'_+|_p(t),\quad \| \Bv^-_t\|_{L^p(\mu_t;M_-)}= |\mu'_-|_p(t).
  \end{align}
  In particular, both $\Bv^+$ and $\Bv^-$ are unique.
\end{corollary}
\begin{proof} We just prove the ``$+$" case since the ``$-$"  case can be derived from a similar argument.
The existence of $\mathbf{v}^+$ follows by Theorem \ref{contequat1} and Proposition \ref{uncontcounequa1}. It remains to show the uniqueness. Suppose that there are two different vector fields $\mathbf{v}^+_t,\mathbf{w}^+_t$ satisfying the continuity equation and the first equality in \eqref{constaid}. Thus,  $\mathbf{y}^+_t:=\frac12(\mathbf{v}^+_t+ \mathbf{w}^+_t)$ also satisfies   \eqref{contequa} and hence, Proposition \ref{coneqbasicproper} yields $|\mu'_+|_p(t)\leq \|\mathbf{y}^+_t\|_{L^p(\mu_t;M_+)}$.
Since $\|\cdot\|_{L^p(\mu_t;M_+)}$ is strictly convex due to the strict convexity of $F$, we have
\[
\|\mathbf{y}^+_t\|_{L^p(\mu_t;M_+)}< \frac12\|\mathbf{v}^+_t\|_{L^p(\mu_t;M_+)}+\frac12\|\mathbf{w}^+_t\|_{L^p(\mu_t;M_+)}=|\mu'_+|_p(t),
\]
 which leads to a contradiction. Therefore, the uniqueness follows.
\end{proof}

\begin{remark}It follows from \eqref{definbv} that
$\Bv^+$ (resp., $\Bv^-$) is constructed from $\eta_+$ (resp., $\eta_-$). In view of Remark \ref{etapmexplain}, $\Bv^-$ is the counterpart of $\Bv^+$ in the reverse Finsler manifold. In particular,  $\Bv^+=\Bv^-$ if $F$ is reversible.
\end{remark}

\section{Generalizations }\label{forresul1t}

Many results in the previous sections are independent of  differential structures of  manifolds.  Hence, we extend them to the nonsmooth setting. The following definition is  a natural generalization of Finsler manifold, which was introduced in \cite{KZ}.

\begin{definition}[\cite{KZ}]\label{thetametricspace}
Let $\Theta:(0,\infty) \rightarrow [1,\infty)$ be a (not necessarily continuous) non-decreasing function.
A triple $(X,\star,d)$ is called a \emph{pointed forward $\Theta$-metric space}
if $(X,d)$ is an asymmetric metric space and $\star$ is a point in $X$ such that
$\lambda_d \left(\overline{B^+_\star(r)}\right) \leq \Theta(r)$ for all $r>0$, where
\[
 \lambda_d\big( \overline{B^+_\star(r)} \big)
 :=\inf\big\{ \lambda\geq1 \,\big|\,
 d(x,y)\leq \lambda d(y,x) \text{ for any } x,y \in \overline{B^+_\star(r)} \big\}.
\]
If there is a constant function $\Theta\equiv\theta$ (i.e., $\lambda_d(X) \le \theta$),
then   $(X,d)$ is called a \emph{$\theta$-metric space}.
\end{definition}
Suppressing $\star$ and $\Theta$ for the sake of simplicity,
we will write $(X,d)$ and call it a \emph{forward metric space}.
Owing to Theorem \ref{reversbilityofW}, every forward   complete Finsler manifold can be viewed as a forward  metric space.
However, the class of forward  metric spaces also contains non-Finslerian examples.
 \begin{example}
Let  $(\mathscr{B},\|\cdot\|)$ be a reflexive Banach space
and $(\mathscr{B}^*,\|\cdot\|_*)$ be its dual space.
Given $\omega \in \mathscr{B}^*$ with $\|\omega\|_*<1$,
 define an asymmetric metric $d_\omega$ on $\mathscr{B}$ by
\[
d_\omega(x,y) :=\|y-x\| +\omega(y-x),\quad \forall\,x,y\in \mathscr{B}.
\]
Thus, $(\mathscr{B},d_\omega)$ is a $\theta$-metric space with $\theta=(1+\|\omega\|_*)/(1-\|\omega\|_*)$.

\end{example}

For  forward  metric spaces,
the backward topology $\mathcal{T}_-$ is weaker than the forward topology $\mathcal{T}_+$ (compared with Theorem \ref{asymmetrtopol}).

\begin{theorem}[\cite{KZ}] \label{topologychara}
Let $(X, d)$ be a  forward  metric space.
Then,
\begin{enumerate}[{\rm (i)}]
\item\label{contin} $\mathcal{T}_- \subset \mathcal{T}_+$  and hence, $d$ is continuous in $\mathcal{T}_+ \times \mathcal{T}_+$;
in particular $(X,\mathcal{T}_+)$ is a Hausdorff space$;$

\item $\mathcal{T}_+$ coincides with the symmetrized topology $\hat{\mathcal{T}}$.
\end{enumerate}
\end{theorem}

\begin{remark}\label{forwardpointspaceandbackwardones}Some more remarks are in order.
\begin{enumerate}[(a)]

\item \label{pfms-b}
If $(X,\star,d)$ is a pointed forward $\Theta$-metric space, then  for every $x\in X$,
the triple $(X,x,d)$ is a pointed forward $\widetilde{\Theta}$-metric space
for $\widetilde{\Theta}(r):=\Theta(d(\star,x)+r)$.
Moreover, if $\diam(X):=\sup_{x,y\in X}d(x,y)<\infty$,
then $(X,d)$ is a $\theta$-metric space with $\theta:=\Theta(\diam(X))$.

\item \label{pfms-c}
One can similarly introduce a \emph{pointed backward $\Theta$-metric space}
 $(X,\star,d)$ by $\lambda_d (\overline{B^-_\star(r)}) \leq \Theta(r)$ for $r>0$.
Note that a pointed forward $\Theta$-metric space may not be
a pointed backward $\widetilde{\Theta}$-metric space for any $\widetilde{\Theta}$;
for example the Funk space defined by \eqref{distFunk}.
Since $(X,\star,d)$ is a pointed backward $\Theta$-metric space
if and only if $(X,\star,\overleftarrow{d})$ is a pointed forward $\Theta$-metric space,
we will focus only on pointed forward $\Theta$-metric spaces.
\end{enumerate}
\end{remark}

Since $\mathcal{T}_-$ is weaker than $\mathcal{T}_+$ for a forward metric space, the backward absolutely continuous curves may be discontinuous. Thus,
Proposition \ref{FACBAC} can be modified as follows.
\begin{proposition}
Let $(X,d)$ be a forward metric space and let $I$ be a bounded closed interval. Thus,
\[
\AC^p(I;X)=\FAC^p(I;X)\subset \BAC^p(I;X),\quad \forall\,p\in [1,\infty].
\]
\end{proposition}
\begin{proof} Without loss of generality, we may suppose that $(X,\star,d)$ is a pointed forward $\Theta$-metric space and $I=[0,1]$.
  Provided $\gamma\in \FAC^p([0,1];X)$, there exists some $f\in L^p([0,1])$ satisfying \eqref{moredfabsc}. Hence,
\[
d(\gamma(0),\gamma(t))\leq \int^t_0 f(s){\dd}s\leq \|f\|^p_{L^p}<\infty\quad \Longrightarrow \quad \gamma([0,1])\subset \overline{B^+_\star(d(\star,\gamma(0))+\|f\|^p_{L^p})},
\]
which implies that for any $[t_1,t_2]\subset [0,1]$,
\[
d(\gamma(t_2),\gamma(t_1))\leq \Theta(d(\star,\gamma(0))+\|f\|^p_{L^p})\,d(\gamma(t_1),\gamma(t_2)) \leq \int^{t_2}_{t_1}g(s){\dd}s,
\]
where $g:=\Theta(d(\star,\gamma(0))+\|f\|^p_{L^p})\,f\in L^p([0,1])$. Hence, $\FAC^p([0,1];X)\subset \BAC^p([0,1];X)$.
\end{proof}
\begin{remark}
If $\gamma\in \BAC^p(I;X)$ is continuous, then $\gamma\in \FAC^p(I;X)$. Factually the compactness of $\gamma(I)$ and Theorem \ref{topologychara}/\eqref{contin} imply $R:=\sup_{x\in \gamma(I)}d(\star,x)<\infty$ and hence, $\gamma(I)\subset \overline{B^+_\star(R)}$.   Thus, a similar argument to the above yields $\gamma\in \FAC^p(I;X)$.
\end{remark}

The concept of metric absolutely continuous curve  in Definition \ref{abthreekinds}/\eqref{metriccontinuous} can be extended naturally to forward metric spaces whereas the definition of naturally absolutely continuous curve cannot due to the lack of differential structures.
Hence, we have the following result, whose proof is given in Appendix \ref{forwauxipro}.
\begin{theorem}\label{basicabsoulcurve}
 Let $(X,d)$ be a separable  forward complete forward metric space and let $I$ a bounded closed interval. Thus,
 \[
  \MAC(I;X)= \AC(I;X)=\FAC(I;X)\subset \BAC(I;X).
 \]
 \end{theorem}

Moreover, one can define the Wasserstein spaces over forward metric spaces  in the same way as in Section \ref{asbwassficcns} (also see \cite{KZ}).
In particular, all the arguments and results in Sections \ref{measfinsler}--\ref{sacws} remain valid for separable forward complete forward metric spaces while the results in Section \ref{Contineqa} can be extended to Minkowski normed spaces by simple modifications. We leave the formulation of such statements to the interested reader.

 \appendix

\section{Complementary results for forward metric spaces}\label{forwauxipro}
In this section, we study variation of   curves in the asymmetric setting.
Let  $(X,d)$ be a separable forward complete forward metric space and let $\gamma:[0,1]\rightarrow X$ be a continuous curve. The curve $\gamma$ is said to be {\it of bounded variation} if
\[
L_d(\gamma):=\sup_{0=t_0\leq t_1\leq \cdots\leq t_n=1 }\sum_{i=0}^{n-1} d(\gamma(t_i),\gamma(t_{i+1}))<\infty.
\]
Owing to  \cite{AGS,KZ,OZ}, $L_d(\gamma)$ is exactly the length of $\gamma$. Hence, $\gamma$ is of bounded variation if and only if it is rectifiable (i.e., the length is finite). In particular, in view of \cite[Lemma 1.1.4]{AGS} and \cite[Theorem 2.7.6]{DYS}, it is not hard to check that
\begin{equation}\label{lengforw}
L_d(\gamma)=\int^1_0 |\gamma'_+|(t){\dd}t,\quad \forall\,\gamma\in \FAC([0,1];X).
\end{equation}

If $\gamma$ is of bounded variation and $J\subset [0,1]$ is an open interval in $\mathbb{R}$, the {\it (pointwise) variation} of $\gamma$ on $J$, say $\Vv(\gamma;J)$, is defined by
\begin{itemize}
\item if $J=(a,b)$, then
$\Vv(\gamma;J):=\sup_{a=t_0\leq t_1\leq \cdots\leq t_n=b }\sum_{i=0}^{n-1} d(\gamma(t_i),\gamma(t_{i+1}))$;

\smallskip

\item if $J$ is a disjoint union of open intervals $J_\alpha$ contained in $[0,1]$, then
$\Vv(\gamma;J):=\sum_\alpha\Vv(\gamma;J_\alpha);$

\smallskip

\item by the continuity of $\gamma$, also set
$\Vv(\gamma;[a,b])=\Vv(\gamma;(a,b])=\Vv(\gamma;[a,b))=\Vv(\gamma;(a,b))$.
\end{itemize}
Finally, the {\it total variation} of $\gamma$ is defined as $\Vv(\gamma;[0,1])$. Proceeding as in (the first part of) the proof of \cite[Theorem 4.4.8]{HKST}, one can show the following result.

\begin{theorem}\label{derivatboundedvariation}
Let $(X,d)$ be a separable forward complete forward metric space and let $\gamma:[0,1]\rightarrow (X,d)$ be a continuous curve of bounded variation. There is a unique Radon measure $\nu_\gamma$ on $[0,1]$ such that $\nu_\gamma(J)=\Vv(\gamma;J)$ for every open interval $J\subset [0,1]$ and the derivative
of $\nu_\gamma$ with respect to   $\mathscr{L}^1$
\begin{equation}\label{lebesgue-radon}
\frac{{\dd}\nu_\gamma}{{\dd}\mathscr{L}^1}(t):=\lim_{\epsilon\rightarrow 0^+}\frac{\nu_\gamma([t-\epsilon,t+\epsilon])}{\mathscr{L}^1([t-\epsilon,t+\epsilon])}
\end{equation}
exists for $\mathscr{L}^1$-a.e. $t\in (0,1)$.
\end{theorem}

\begin{proposition}\label{basicmaac}
Let $(X,d)$ be a separable forward complete forward metric space.
Thus, $\gamma\in \MAC([0,1];X)$ if and only if the associated Radon measure $\nu_\gamma$ is absolutely continuous with respect to $\mathscr{L}^1$. In particular, if $\gamma$ is metric absolutely continuous, then $\gamma\in \FAC([0,1];X)$ and in particular, for $\mathscr{L}^1$-a.e. $t\in (0,1)$,
\[
\frac{{\dd}\nu_\gamma}{{\dd}\mathscr{L}^1}(t)=|\gamma'_+|(t).
\]
 \end{proposition}
\begin{proof}
Suppose $\gamma\in \MAC([0,1];X)$. In view of Definition \ref{abthreekinds}/\eqref{metriccontinuous}, it is not hard to show that $\gamma$ is a continuous curve of bounded variation.  Theorem \ref{derivatboundedvariation} then implies that $\rho(t):=\frac{{\dd}\nu_\gamma}{{\dd}\mathscr{L}^1}(t)$ exists for $\mathscr{L}^1$-a.e. $t\in (0,1)$. Thus, by the Lebesgue-Radon-Nikodym theorem (cf.~\cite[p.\,82]{HKST}), we have $\nu_\gamma=\rho\mathscr{L}^1+(\nu_\gamma)^s$, where $(\nu_\gamma)^s$ denotes the singular part of $\nu_\gamma$ such that $(\nu_\gamma)^s\bot \rho\mathscr{L}^1$. In particular, $(\nu_\gamma)^s$ is finite and nonnegative. We now claim that $(\nu_\gamma)^s(\{t\})=0$ for every $t\in [0,1]$. In fact,
if there is some $t_0\in [0,1]$ with $(\nu_\gamma)^s(\{t_0\})>0$,  we have
\[
0<(\nu_\gamma)^s(\{t_0\})\leq  (\nu_\gamma)^s((t_0-\delta,t_0+\delta))\leq \nu_\gamma((t_0-\delta,t_0+\delta))=\Vv(\gamma;(t_0-\delta,t_0+\delta)),\quad \forall\,\delta>0,
\]
which is contrary to  $\gamma\in \MAC([0,1];X)$ because $\lim_{\delta\rightarrow 0} \Vv(\gamma;(t_0-\delta,t_0+\delta))=0$ (see Definition \ref{abthreekinds}/\eqref{metriccontinuous}). So the claim is true, which furnishes
\begin{equation}\label{finiteiszero}
\nu_\gamma(F)=0\quad \text{for any finite set $F\subset [0,1]$}.
\end{equation}

Now we show that $ (\nu_\gamma)^s$ is a null measure. If not, there is a  Lebesgue null set $E\subset [0,1]$ with $(\nu_\gamma)^s(E)>0$ due to $(\nu_\gamma)^s\bot \rho\mathscr{L}^1$.  By Definition \ref{abthreekinds}/\eqref{metriccontinuous}, for any $\varepsilon\in (0,(\nu_\gamma)^s(E)/2)$, there exists $\delta=\delta(\varepsilon)>0$ such that
\begin{equation}\label{controllbounded}
\sum_{i=1}^n d(\gamma(a_i), \gamma(b_i)) <\varepsilon<\frac12{(\nu_\gamma)^s(E)},
\end{equation}
whenever $\{(a_i, b_i)\}^n_{i=1}$ are nonoverlapping subintervals of $[0,1]$ with
 $\sum_{i=1}^n (b_i-a_i) <\delta$. On the other hand, $\mathscr{L}^1(E)=0$ yields
a countable sequence  of nonoverlapping subintervals $\{(c_k,d_k)\}_k\subset [0,1]$ such that $E\subset \cup_{k}[c_k,d_k]$ and $\sum_{k}(d_k-c_k)<\delta$. By the definition of $\Vv(\gamma;\cdot)$ and \eqref{controllbounded}, a direct argument yields
\[
\sum_{k=1}^\infty\Vv(\gamma;(c_k,d_k))=\lim_{N\rightarrow \infty}\sum_{k=1}^N\Vv(\gamma;(c_k,d_k))=\lim_{N\rightarrow \infty}\Vv\left(\gamma; \bigcup_{k=1}^N (c_k,d_k)  \right)\leq \varepsilon< \frac12(\nu_\gamma)^s(E).
\]
However, owing to  \eqref{finiteiszero} and Theorem \ref{derivatboundedvariation}, we also have
\begin{align*}
0<(\nu_\gamma)^s(E)\leq \nu_\gamma(E)\leq \sum_{k=1}^\infty\nu_\gamma([c_k,d_k])=\sum_{k=1}^\infty\nu_\gamma((c_k,d_k))
=\sum_{k=1}^\infty\Vv(\gamma;(c_k,d_k))< \frac12(\nu_\gamma)^s(E),
\end{align*}
which leads to a contradiction. Thus,
$(\nu_\gamma)^s$ must be a null measure and hence, $\nu_\gamma$ is absolutely continuous with respect to $\mathscr{L}^1$.

On the other hand, if $\nu_\gamma$ is absolutely continuous with respect to $\mathscr{L}^1$, then $\nu_\gamma=\rho\mathscr{L}^1$ and hence,
\begin{equation*}\label{forwardequat}
d(\gamma(t_1),\gamma(t_2))\leq \Vv(\gamma;(t_1,t_2))= \nu_\gamma((t_1,t_2))=\int^{t_2}_{t_1}\rho(s){\dd}s,
\end{equation*}
which implies $\gamma\in \FAC([0,1];X)\subset\MAC([0,1];X)$. Moreover, since $\nu_\gamma((t_1,t_2))=\Vv(\gamma;(t_1,t_2))$ is the length of $\gamma|_{[t_1,t_2]}$,   we obtain  $\rho(t)=|\gamma'_+|(t)$ by  \eqref{lengforw}.
\end{proof}

\begin{proof}[Proof of Theorem \ref{basicabsoulcurve}]  The definitions indicate
 $\FAC([0,1];X)\subset \MAC([0,1];X)$. On the other hand, if $\gamma\in \MAC([0,1];X)$, then Proposition \ref{basicmaac} implies $\gamma\in \FAC([0,1];X)$, which concludes the proof.
\end{proof}

\end{document}